\documentclass{article}

\usepackage{longtable}
\usepackage{fixmath}
\usepackage{mathdots}
\usepackage{environ}
\usepackage{xcolor}
\usepackage{caption}
\usepackage[a4paper,left=1in,right=1in,top=1in,bottom=1in]{geometry}
\usepackage{amsfonts,amsmath,amssymb,amsthm}
\usepackage{hyperref}
\usepackage{graphicx}
\usepackage[english]{babel}
\usepackage[utf8]{inputenc}
\usepackage{csquotes,xpatch}

\MakeOuterQuote{"}

\usepackage{times}
\usepackage{enumitem}
\usepackage{setspace}
\usepackage{biblatex}
\newcommand {\mat}[1] {\left[\begin{array}{#1}}
\newcommand {\rix}{\end{array}\right]}

\graphicspath{{images/}}

\newtheorem{theorem}{Theorem}[section]

\newtheorem{definition}[theorem]{Definition}
\newtheorem{example}[theorem]{Example}

\newtheorem{lemma}[theorem]{Lemma}

\newtheorem{remark}[theorem]{Remark}

\title{Structured Linearizations of Structured Rational Matrices}

\author{
Avisek Bist \thanks{Department of Mathematics, Sikkim University, Sikkim-737102, India, ({\tt avisek.bista@gmail.com})}  \and
Namita Behera \thanks{Corresponding author\\ Department of Mathematics, Sikkim University, Sikkim-737102, India, ({\tt nbehera@cus.ac.in}, niku.namita@gmail.com).}} 

\date{}

\begin{document}

\maketitle

\begin{abstract}
Numerical computations involving rational matrices often benefit from preserving underlying matrix structures such as symmetry, Hermitian properties, or sparsity that reflect physical, geometric, or algebraic characteristics of the system. Maintaining such structures enhances stability, accuracy, and efficiency. Linearization, a technique that reformulates rational matrix problems as generalized eigenvalue problems (GEPs) of larger matrices, is widely used but does not automatically retain structure. In this chapter, we focus on structured linearizations, which preserve both the spectral information of the original rational matrix and its intrinsic structural properties. To achieve this, we present the construction of a family of linearizations called generalized Fiedler pencils with repetition (GFPR), which we prove to be valid linearizations for rational matrices. Moreover, we demonstrate that the GFPR family serves as a versatile framework for generating structured linearizations, specifically symmetric, skew-symmetric, $T$-even, and $T$-odd linearizations, provided the original rational matrix exhibits the corresponding structure. These structured linearizations facilitate the use of specialized, structure-preserving algorithms, reduce numerical errors, and yield physically meaningful solutions in applications.
\end{abstract}

{\bf Keywords: }
Structured rational matrix, system matrix, matrix polynomial, eigenvalues, eigenvector, Fiedler pencil, linearization

{\bf AMS subject classifications:} 
65F15, 15A57, 15A18, 65F35

\section{Introduction}

Structured rational matrices, whose entries are rational functions with additional properties such as symmetry or Hermitian structure, appear naturally in applied mathematics, control theory, and systems theory, often describing frequency-dependent dynamics. Preserving this structure in computations is essential for numerical stability, accuracy, and efficiency.

A common approach is linearization, which converts a rational matrix problem into an equivalent generalized eigenvalue problem (GEP) of larger dimension. A structured linearization is one that retains not only the spectral information of the original rational matrix but also its intrinsic algebraic structure, enabling the use of structure-preserving algorithms and physically meaningful results.

For rational matrix functions $G(\lambda)$ arising as transfer functions of LTI systems, the Rosenbrock system matrix provides a natural bridge from state-space realizations to pencils. Rosenbrock-based (Fiedler-like) pencils have been developed to linearize such rational matrix functions, allowing eigenvector recovery and the handling of minimal indices. These constructions, starting from controllable and observable realizations, laid the groundwork for structured linearizations \cite{AB2016}. Analogous frameworks have been extended to multivariable systems, covering both square and rectangular cases \cite{BB2022,BBM2024}.

A principal class of structure-preserving linearizations is the vector space $\mathbb{DL}(G)$ introduced in Chapter 4, whose pencils are built from the coefficient matrices of $A(\lambda)$ and $D(\lambda)$. While these provide linearizations for regular $G(\lambda)$, they generally do not yield strong linearizations in the singular case.

When $G(\lambda)$ possesses algebraic symmetry, structured strong linearizations are designed to inherit the same symmetry, e.g., symmetric, Hermitian, or para-Hermitian---thereby preserving eigenvalue symmetries and supporting specialized algorithms. Recent work has produced families of Fiedler-like strong linearizations that preserve structure and maintain properties such as the index in symmetric settings, along with strongly minimal self-conjugate linearizations for Hermitian and para-Hrmitian classes \cite{DA2022}.

\section{Generalized Fiedler pencils with repetition}

We focus on rational matrix functions expressed in the realization form
\begin{equation}\label{trfunction_chap05}
G(\lambda) = C A(\lambda)^{-1} B + D(\lambda) \in \mathbb{C}(\lambda)^{r \times r},
\end{equation}
which arises from the linear time-invariant system
\begin{equation}\label{my_lti_chap05}
\begin{aligned}
       A\left(\frac{d}{dt}\right) x(t) &= B u(t), \\
       y(t) &= C x(t) + D\left(\frac{d}{dt}\right) u(t),
\end{aligned}
\end{equation}
where $A(\lambda) = \sum_{j=0}^{m}\lambda^{j}A_j \in \mathbb{C}[\lambda]^{n \times n}$ is a regular matrix polynomial of degree $m$ and $D(\lambda) = \sum_{j=0}^{k}\lambda^{j}D_j \in \mathbb{C}[\lambda]^{r \times r}$ is a matrix polynomial of degree $k$, with $C \in \mathbb{C}^{r \times n}$, $B \in \mathbb{C}^{n \times r}$. The associated system matrix is given by
\begin{equation}\label{sysmatrix_chap05}
\mathcal{S}(\lambda) = \left[\begin{array}{c|c}
                            A(\lambda) & -B \\ \hline
                            C & D (\lambda) \\
                       \end{array}\right] 
                       \in \mathbb{C}[\lambda]^{(n+r)\times (n+r)}.
\end{equation}

  Next we define elementary matrices for the polynomial matrices $A(\lambda)$ and $D(\lambda)$ which is the generalization of Fiedler matrices defined in \cite{TDM2010}.

\subsection{Fiedler matrices}

Given a matrix $P\in\mathbb{C}^{n\times n}$, the family of elementary block matrices for $A(\lambda)$ are defined as \cite{BDFR2015},

\[
M_{0} (P) = \left[
                   \begin{array}{@{}cc@{}}
                     I_{(m-1)n} &  \\
                      & P \\
                   \end{array}
                 \right],~   M_{i} (P) = \left[
                 \begin{array}{@{}cccc@{}}
                 I_{(m-i-1)n} &  & & \\
                 & P & I_{n} & \\
                 & I_{n} & 0  & \\
                 &   &   & I_{(i-1)n}\\
                 \end{array}
                 \right] ~\mbox{for}~ i= 1,2,\ldots, m-1,
\]
\[
M_{-m}(P) = \left[
\begin{array}{@{}cc@{}}
P &  \\
& I_{(m-1)n} \\
\end{array}
\right],  ~ M_{-i}(P) = \left[
  \begin{array}{@{}c@{\;}ccc@{}}
    I_{(m-i-1)n} &  & & \\
     & 0 & I_{n} & \\
     & I_{n} & P  &  \\
     &   &   & I_{(i-1)n}
  \end{array}\right] ~\mbox{for}~ i= 1,2,\ldots, m-1.
\]
For matrices $P, Q \in \mathbb{C}^{n \times n}$, the matrices constructed above have the following properties \cite{BDFR2015}.
\begin{itemize}
\item For any $i=1,\dots,m-1$, the matrices $M_i(P)$ and $M_{-i}(P)$ are non-singular and
\[
M_i(P) \, M_{-i}(-P) = I_{mn}.
\]

\item The matrices $M_0(P)$ and $M_{-m}(P)$ are non-singular whenever $P$ is.

\item If $\bigl|\, |i| - |j| \,\bigr| > 1$, then $M_i(P)$ and $M_j(Q)$ commute:
\[
M_i(P)M_j(Q) = M_j(Q)M_i(P).
\]
\end{itemize}
For $i=-m,\ldots,-1,0,1,\ldots,m-1$, and for $A(\lambda)$ we define\cite{BDFR2015}
 \[
 M_{i}^{A}   =
     \left\{ \begin{array}{ll}
              M_i(-A_i) & \mbox{for}~ i= 0,1,\ldots,m-1\\
              M_i(A_{-i}) & \mbox{if}~ i=-1,-2,\ldots,-m. 
              \end{array}
        \right.
\]
Then matrices $ M_{i}^{A} $ for $i=-m,\ldots,-1,0,1,\ldots,m-1$ are exactly the Fiedler matrices associated with $A(\lambda) $, see \cite{TDM2010}.

Similarly, for any $P \in \mathbb{C}^{r \times r}$, we define the elementary block matrices for $D(\lambda)$ as~\cite{BDFR2015}
\[
N_{0}(P) = \left[\begin{array}{@{}cc@{}}
                     I_{(k-1)r} &  \\
                      & P \\
                   \end{array}\right],~   
N_{i}(P) = \left[\begin{array}{@{}cccc@{}}
                 I_{(k-i-1)r} &  & & \\
                 & P & I_{r} & \\
                 & I_{r} & 0  & \\
                 &   &   & I_{(i-1)r}\\
            \end{array}\right] ~\mbox{for}~ i= 1,\ldots,k-1,
\]
\[
N_{-k}(P) = \left[\begin{array}{@{}cc@{}}
                   P &  \\
                     & I_{(k-1)r} \\
             \end{array}\right],  ~ 
N_{-i}(P) = \left[\begin{array}{@{}c@{\;}ccc@{}}
               I_{(k-i-1)r} &  & & \\
                 & 0 & I_{r} & \\
                 & I_{r} & P  &  \\
                 &   &   & I_{(i-1)r}\\
              \end{array}\right] ~\mbox{for}~ i= 1,\ldots, k-1.
\]
Note that, for any  matrix $P \in \mathbb{C}^{r \times r}$ and $i=1,\ldots,k-1$, $N_{i}(P)$ and $N_{-i}(P)$ are non-singular  with $(N_{i}(P))N_{-i}(-P)=I_{rk}$ while $N_{0} (P)$ and $N_{-k}(P)$ are non-singular exactly when $P$ is non-singular. Further, $N_{i}(P)N_{j}(Q)$ commute for any $P, Q \in \mathbb{C}^{r \times r}$ if $||i|-|j|| > 1$,  see \cite{BDFR2015}. For $i=-k,\ldots,-1,0,1,\ldots,k-1$ and $D(\lambda)$ we define \cite{BDFR2015}
 \[
 N_{i}^{D}   =
     \left\{ \begin{array}{ll}
              N_i(-D_i) & \mbox{for}~ i= 0,1,\ldots,k-1\\
              N_i(D_{-i}) & \mbox{if}~ i=-1,-2,\ldots,-k.
              \end{array}
        \right.
\]
In such a case $ N_{i}^{D} $ for $i=-k,\dots,-1,0,1,\ldots,k-1$ are exactly the Fiedler matrices associated with $D(\lambda)$ see \cite{TDM2010}.

\subsection{Index tuples}
We now introduce index tuples and their characteristics. These are important for constructing generalized Fiedler pencils with repetitions.

\begin{definition}[Permutations \& Sub-permutations, \cite{AB2018}]
Consider a finite set $H$. A permutation of $H$ is a bijective function $\gamma : H \rightarrow H$. If $G\subseteq H$ and $\delta$ is a permutation of $G$ then $\delta$ is said to be a sub-permutation of $\gamma$.  

Let $\delta_{1}$ and $\delta_{2}$ be two sub-permutations of $\gamma$ associated with the subsets $G_{1}$ and $G_{2}$ of $H$. If $\delta_{1}\cap\delta_{2}=G_{1}\cap G_{2}=\emptyset$, then the sub-permutations $\delta_{1}$ and $\delta_{2}$ are said to be disjoint. Further if for any sub-permutations $\delta_{1}$ and $\delta_{2}$, $ \gamma = (\delta_{1} , \delta_{2})$ constitutes a permutation of $H$, then $\delta_{1}$ and $\delta_{2}$ are called a partition of $\gamma.$
\end{definition}

It is to be noted here that the empty permutation is denoted by the symbol $\emptyset$.

\begin{definition} [Index Tuple, \cite{AB2018}] 
An index tuple is an ordered tuple $\mathbold{t} =(t_1 , t_2 ,\ldots , t_p)$ with entries from $\mathbb{Z}$, that is, for $i=1,2,\ldots,p$, $t_{i}\in\mathbb{Z}$. Following operations are defined on the index tuples:
\begin{itemize}
\item $- \mathbold{t} = ( -t_1, -t_2, \ldots , -t_p)$,

\item $rev(\mathbold{t})=( t_p ,  \ldots, t_2, t_1),$ 

\item for any $k\in\mathbb{Z}$, $\mathbold{t} + k= ( t_1 +k , t_2+k, \ldots , t_p +k)$,  

\item for two index tuples $\mathbold{t} = ( t_1  , \ldots , t_p )$ and $\mathbold{s} = ( s_1  , \ldots , s_q)$, we define  $\mathbold{t} \cup \mathbold{s} = (\mathbold{t} ,\mathbold{s}) = ( t_1  , \ldots , t_p , s_1  , \ldots , s_q)$.
\end{itemize}
\end{definition}

For an index tuple we next define its SIP and csf. These two notions will be widely used in the following sections.

\begin{definition} [\cite{BT2014,VA2011}] Let $h\in\mathbb{Z}$ and $h\geq0$. Consider an index tuple $\mathbold{\gamma} = (j_1, j_2, \ldots , j_p)$ with entries from the set $\{0,1,\ldots, h\}$.
\begin{itemize}
\item  If for any two indices $j_{s}, j_{t}$ in $\gamma$ where $1 \leq s<t \leq p$ and $j_s = j_t,$ we can find an index $j_r = j_s +1$ such that $s < r <t$, then we say that $\mathbold{\gamma}$ satisfies the Successor Infix Property (SIP).  

If an index tuple $\mathbold{\gamma}$ with entries from  $\{-h,-h+1,\ldots ,-1 \}$ is such that $\mathbold{\gamma}+h$ satisfies the SIP, then we say that $\mathbold{\gamma}$ satisfies the SIP.

\item The column standard form of $\mathbold{\gamma}$ denoted by $csf(\mathbold{\gamma})$ is
\[
\mathbold{\gamma}= ( m_{q}:n_{q}, m_{q-1}: n_{q-1}, \ldots , m_1 : n_1),
\]
with $0 \leq n_1 < \cdots < n_s \leq h $ and $0 \leq m_{i} \leq n_{i} ,$ for all $i =1 , \ldots ,q.$ 

If and index tuple $\mathbold{\gamma}$ with entries form $\{-h,-h+1,\ldots ,-1 \}$ be such that $\mathbold{\gamma} + h$ is in column standard form then $\mathbold{\gamma}$ is said to be in column standard.
\end{itemize}
\end{definition}

\begin{example}
{\rm
   $\mathbold{\gamma}=(0,1,0,2,1)$ is an index tuple satisfying the SIP property.}
\end{example}

\begin{definition} [\cite{BDFR2015}]  Consider two index tuples $\mathbold{\gamma}_{1}$ and $\mathbold{\gamma}_{2}$. If either $\mathbold{\gamma}_{1} = \mathbold{\gamma}_{2}$ or $\mathbold{\gamma}_{1}$ can be obtained from $\mathbold{\gamma}_{2}$ by removing some entries of $\mathbold{\gamma}_{2}$ then we say that $\mathbold{\gamma}_{1}$ is a subtuple of $\mathbold{\gamma}_{2}$.
\end{definition}

\begin{example}
{\rm
Let $\mathbold{\gamma} = ( 0,1,0,2,1)$. Then  $(0,1,2)$ is a subtuple of $\mathbold{\gamma}$ whereas $ ( 2,1,0)$ is not a subtuple of $\mathbold{\gamma}.$}
\end{example}

The following definitions of consecutive inversions and consecutive consecutions are an important property of an index tuple. They will be used widely in the coming sections.

\begin{definition}[Consecutions \& inversions, \cite{DA2019Automatic}] \label{coninvoftuple_RCh4} 
Consider an index tuple $\mathbold{\gamma}$ with entries from $\{ 0,1,\ldots, p\}$. If $r \in \mathbold{\gamma}$ be such that $ ( r , r+1 , \ldots, r+s)$ is a subtuple of $\mathbold{\gamma}$ but $ ( r , r+1 , \ldots, r+s, r+s+1)$ is not a subtuple of $\mathbold{\gamma}$ then we say that $\mathbold{\gamma}$ has $s$ consecutive consecutions at $r$. We use $ c_{r}(\mathbold{\gamma})$ to denote the number of consecutive consecutions of $\mathbold{\gamma}$ at $r$. 

Similarly, if $ (r+s, \ldots, r+1, r)$ is a subtuple of $\mathbold{\gamma}$ and $ (r+s+1, r+s, \ldots, r+1, r)$ is not a subtuple of $\mathbold{\gamma}$ then $\mathbold{\gamma}$ is said to have $s$ consecutive inversions at $r$. We use $i_{r}(\mathbold{\gamma})$ to denote the number of consecutive inversions of $\mathbold{\gamma}$ at $r$.

If $k \in \{ 0,1,\ldots, p\} $ and $ k \notin \mathbold{\gamma}$, we define $ c_k(\mathbold{\gamma}) =-1$ and $i_k(\mathbold{\gamma}) = -1.$
\end{definition}

\begin{example}
{\rm
Consider $\mathbold{\gamma} = ( 3,4,1,6,2,3,1,2,4,5,2)$ with entries from $\{0,1,2,3,4,5,6\}$. Then 
\begin{itemize}
\item $ c_{1}(\mathbold{\gamma}) = 4$ as $(1,2,3,4,5)$ is a subtuple of $\mathbold{\gamma}$ and $(1,2,3,4,5,6)$ is not a subtuple of $\mathbold{\gamma}$. 

\item Similarly, $ i_{2}(\mathbold{\gamma}) = 2$ as $(4,3,2)$ is a subtuple of $\mathbold{\gamma}$ and $(5,4,3,2)$ is not a subtuple of $\mathbold{\gamma}$.

\item As $0 \notin \mathbold{\gamma}$ we have  $c_{0}(\mathbold{\gamma}) = -1 $ and $i_{0}(\mathbold{\gamma}) =-1.$
\end{itemize}}
\end{example}


We construct a new class of Fiedler-like pencils called generalized Fiedler penicls with repitition (GFPRs) for rational matrices. This is how we go about it.

\begin{definition}[\cite{BDFR2015}, Matrix Assignment]
Consider and index tuple $ \mathbold{p}=(p_{1},p_{2},\ldots,p_{k})$ with entries from $\{-m,\ldots,-1,0,1,\ldots,m-1\}$ and a tuple of $n\times n$ matrices $P=\left(P_{1},P_{2},\ldots,P_{k}\right)$. The matrix assignment of $P$ for $\mathbold{p}$ is the product $M_{\mathbold{p}}\left(P\right)= M_{p_{1}}\left(P_{1}\right)M_{p_{2}}\left(P_{2}\right)\cdots M_{p_{k}}\left(P_{k}\right)$ and the matrix $P_{j}$ is said to be assigned to position $i$ in $\mathbold{p}$.

If the matrices assigned by $P$ to the positions $-m$ and $0$ are non-singular, the matrix assignment $P$ is nonsingular. Further, we define $rev\left(P\right)=\left(P_{k},\ldots,P_{2},P_{1}\right)$.    

For an index tuple $\mathbold{p}$, the trivial matrix assignment associated with the matrix polynomial $A(\lambda)$ is the one where $M_{p_{i}}(P_{i})=M^{A}_{p_{i}}$ for $i=1,2,\ldots,k$. Further, we define $M_{\mathbold{p}}^{A}=M^{A}_{p_{1}}M^{A}_{p_{2}}\cdots M^{A}_{p_{k}}$.
\end{definition}

Next we construct the GFPR for $G(\lambda)$ defined in (\ref{trfunction_chap05}) where $A(\lambda)\in\mathbb{C}[\lambda]^{n\times n}$ is regular with degree $m$, $D(\lambda)\in\mathbb{C}[\lambda]^{r\times r}$ with degree $k$, $C\in\mathbb{C}^{r\times n}$ and $B\in\mathbb{C}^{n\times r}$.

\begin{definition}[GFPR of $G(\lambda)$]\label{gfprdefinition}
    Let $h\in\{0,1,\ldots, m-1\}$. Let $\mathbold{\sigma}$ be a permutation of $\{0,1,\ldots,h\}$ and $\mathbold{\tau}$ be a permutation  $\{-m,-m+1,\ldots,-h-1\}$. Consider the index tuples $\mathbold{\sigma}_{1}$ and $\mathbold{\sigma}_{2}$ with entries from $\{0,1,\dots,h-1\}$ with  $(\mathbold{\sigma}_{1},\mathbold{\sigma},\mathbold{\sigma}_{2})$ satisfying the SIP. Similarly, consider index tuples $\mathbold{\tau}_{1}$ and $\mathbold{\tau}_{2}$ with entries from $\{-m,-m+1,\ldots,-h-2\}$ with $(\mathbold{\tau}_{1},\mathbold{\tau},\mathbold{\tau}_{2})$ satisfying the SIP. Consider the matrix assignments $X^{A}_{1}$, $X^{A}_{2}$, $Y^{A}_{1}$ and $Y^{A}_{2}$ for $\mathbold{\sigma}_{1},\mathbold{\sigma}_{2},\mathbold{\tau}_{1}$ and $\mathbold{\tau}_{2}$, respectively. Let $\ell\in\{0,1,\ldots, k-1\}$, and let $\mathbold{\gamma}$ be a permutation of $\{0,1,\dots,\ell\}$ and $\mathbold{\delta}$ be a permutation of $\{-k,-k+1,\ldots,-\ell-1\}$. consider the index tuples $\mathbold{\gamma}_{1}$ and $\mathbold{\gamma}_{2}$ with entries from $\{0,1,\ldots,\ell-1\}$ with $(\mathbold{\gamma}_{1},\mathbold{\gamma},\mathbold{\gamma}_{2})$ satisfying the SIP. Similarly, consider index tuples $\mathbold{\delta}_{1}$ and $\mathbold{\delta}_{2}$ with entries from $\{-k,-k+1,\ldots,-\ell-2\}$ with $(\mathbold{\delta}_{1},\mathbold{\delta},\mathbold{\delta}_{2})$ satisfying the SIP. Let $\mathbold{\gamma}_{1},\mathbold{\gamma}_{2},\mathbold{\delta}_{1}$ and $\mathbold{\delta}_{2}$ have matrix assignments $X^{D}_{1}$, $X^{D}_{2}$, $Y^{D}_{1}$ and $Y^{D}_{2}$ respectively. The pencil,
    \begin{equation}
     \mathbb{L}(\lambda)=
        \left[\begin{array}{c|c}
            L_{A}(\lambda) & e_{m-i_{0}(\mathbold{\sigma}_{1},\mathbold{\sigma})}e^{T}_{k-c_{0}(\mathbold{\gamma},\mathbold{\gamma}_{2})}\otimes -B \\ \hline
            e_{k-i_{0}(\mathbold{\gamma}_{1},\mathbold{\gamma})}e^{T}_{m-c_{0}(\mathbold{\sigma},\mathbold{\sigma}_{2})} & L_{D}(\lambda)
        \end{array}\right],
    \end{equation}
    where $L_{A}(\lambda)=M_{(\mathbold{\tau_{1}},\mathbold{\sigma}_{1})}(Y^{A}_{1},X^{A}_{1})\left(\lambda M^{A}_{\mathbold{\tau}}-M^{A}_{\mathbold{\sigma}}\right)M_{(\mathbold{\sigma}_{2},\mathbold{\tau}_{2})}(X^{A}_{2},Y^{A}_{2})$ and \\ $L_{D}(\lambda)=N_{(\mathbold{\delta}_{1},\mathbold{\gamma}_{1})}(Y^{D}_{1},X^{D}_{1})(\lambda N^{D}_{\mathbold{\delta}}-N^{D}_{\mathbold{\gamma}})N_{(\mathbold{\gamma}_{2},\mathbold{\delta}_{2})}(X^{D}_{2},Y^{D}_{2})$ is called a generalized Fiedler pencil with repeatition (GFPR) of $G(\lambda)$. We also refer to $\mathbb{L}(\lambda)$ as a GFPR of $\mathcal{S}(\lambda)$.
\end{definition}

\begin{remark}
{\rm
    Note that, the pencils $L_{A}(\lambda)$ and $L_{D}(\lambda)$ are the GFPRs of $A(\lambda)$ and $D(\lambda)$, respectively.}
\end{remark}

\begin{example}
    Let $G(\lambda)=D(\lambda)+CA(\lambda)^{-1}B\in\mathbb{C}(\lambda)^{r\times r}$ where $A(\lambda)=\lambda^{5}A_{5}+\lambda^{4}A_{4}+\lambda^{3}A_{3}+\lambda^{2}A_{2}+\lambda A_{1}+A_{0}\in\mathbb{C}[\lambda]^{n\times n}$, $D(\lambda)=\lambda^{4}D_{4}+\lambda^{3}D_{3}+\lambda^{2}D_{2}+\lambda D_{1}+D_{0}\in\mathbb{C}[\lambda]^{r\times r}$, $C\in\mathbb{C}^{r\times n}$, and $B\in\mathbb{C}^{n\times r}$. If we cosider $h=2$, $\mathbold{\sigma}=(1,0,2)$, $\mathbold{\tau}=(-3,-4,-5)$, $\mathbold{\sigma}_{1}=(0)$, $\mathbold{\sigma}_{2}=(1)$, $\mathbold{\tau}_{1}=(-5)$, and $\mathbold{\tau}_{2}=\phi$ and $\ell=3$, $\mathbold{\gamma}=(1,2,3,0)$, $\mathbold{\delta}=(-4)$, $\mathbold{\gamma}_{1}=\mathbold{\tau}_{1}=\mathbold{\tau}_{2}=\phi$, and $\mathbold{\gamma}_{2}=(2,1)$. Then, GFPR of $G(\lambda)$ is given by,
    \begin{align*}
    &\mathbb{L}(\lambda)
    =\left[\begin{array}{c|c}
            L_{A}(\lambda) & e_{m-i_{0}(\mathbold{\sigma}_{1},\mathbold{\sigma})}e^{T}_{k-c_{0}(\mathbold{\gamma},\mathbold{\gamma}_{2})}\otimes -B \\ \hline
            e_{k-i_{0}(\mathbold{\gamma}_{1},\mathbold{\gamma})}e^{T}_{m-c_{0}(\mathbold{\sigma},\mathbold{\sigma}_{2})} & L_{D}(\lambda)
        \end{array}\right]\\
    &=\left[\begin{array}{c|c}
            M_{(-5,0)}(Y_{1},X_{1})\left(\lambda M^{A}_{(-3,-4,-5)}-M^{A}_{(1,0,2)}\right)M_{1}(X_{2}) & e_{5-i_{0}(0,1,0,2)}e^{T}_{4-c_{0}(1,2,3,0,2,1)}\otimes -B \\ \hline
            e_{4-i_{0}(1,2,3,0)}e^{T}_{5-c_{0}(1,0,2,1)}\otimes C & (\lambda N^{D}_{-4}-N^{D}_{(1,2,3,0)})N_{(2,1)}(X,Y)
        \end{array}\right]
\end{align*}

\[=\left[\begin{array}{ccccc|cccc}
        -Y_{1} & \lambda Y_{1} & 0 & 0 & 0 & 0 & 0 & 0 & 0 \\
        0 & -I_{n} & \lambda I_{n} & 0 & 0 & 0 & 0 & 0 & 0 \\
        \lambda A_{5} & \lambda A_{4} & \lambda A_{3}+A_{2} & -X_{2} & -I_{n} & 0 & 0 & 0 & 0 \\
        0 & 0 & A_{1} & \lambda X_{2}+A_{0} & \lambda I_{n} & 0 & 0 & -B & 0\\
        0 & 0 & -X_{1} & \lambda I_{n} & 0 & 0 & 0 & 0 & 0 \\ \hline
        0 & 0 & 0 & 0 & 0 & \lambda D_{4} + D_{3} & -X & -Y & -I_{n} \\
        0 & 0 & 0 & 0 & 0 & D_{2} & \lambda X- I_{n} & \lambda Y & \lambda I_{n} \\
        0 & 0 & 0 & C & 0 & D_{1} & \lambda I_{n} & A_{0} & 0 \\
        0 & 0 & 0 & 0 & 0 & -I_{n} & 0 & \lambda I_{n} & 0 
      \end{array}\right].
\]
\end{example}

\subsection{GFPRs are linearizations}
Now we focus on showing that the GFPRs defined in 
Definition~\ref{gfprdefinition} are linearizations of $G(\lambda). $
We state here a few important results and theorems before going to the main proof.

\begin{lemma}[\cite{TDM2010}]\label{unimodforfiedler}
Let $\mathbold{\alpha}$ be a permutation of $\{0,1,\ldots,m-1\}$ and let $L_{A}(\lambda)=\lambda M^{A}_{-m}-M^{A}_{\mathbold{\alpha}}$ be the Fiedler pencil of $A(\lambda)$ corresponding to $\mathbold{\alpha}$ Then, $L_{A}(\lambda)$ is a linearization of $A(\lambda)$. Hence, there exist unimodular matrix polynomials $U_{A}(\lambda)$ and $V_{A}(\lambda)$ such that,
\[
U_{A}(\lambda)L_{A}(\lambda) V_{A}\lambda)
     =\left[\begin{array}{c|c}
            I_{(m-1)n} &  \\\hline
               & A(\lambda)
      \end{array}\right].
\]
Further,
\begin{equation}
\begin{aligned}\label{aunimodular}
   U_{A}\left(\lambda)^{-1} (e_{m}\otimes I_{n}\right) 
   &=\begin{cases}
       e_{m}\otimes I_{n} \,\,\,\,\ & \text{if $c_{0}(\mathbold{\alpha})>0$}\\
       e_{m-i_{0}(\mathbold{\alpha)}}\otimes I_{n} & \text{if $c_{0}(\mathbold{\alpha})=0$}
   \end{cases}\\
   \left(e^{T}_{m}\otimes I_{n}\right) V_{A}^{-1}(\lambda) &= e^{T}_{m-c_{0}(\mathbold{\alpha)}}\otimes I_{n}.
\end{aligned}
\end{equation}
\end{lemma}

Similar result holds for the matrix polynomial $D(\lambda)$. That is, if $\mathbold{\beta}$ is a permutation of $\{0,1,\ldots,k-1\}$  and $L_{D}(\lambda)=\lambda N^{D}_{-k}-N^{D}_{\mathbold{\beta}}$ is a Fiedler pencil of $D(\lambda)$ associated with $\mathbold{\beta}$ then $L_{D}(\lambda)$ is a linearization of $D(\lambda)$. Hence, there exist unimodular matrix polynomials $U_{D}(\lambda)$ and $V_{D}(\lambda)$ such that,
\[
U_{D}(\lambda)L_{D}(\lambda) V_{D}\lambda)
     =\left[\begin{array}{c|c}
            I_{(k-1)r} &  \\\hline
               & D(\lambda)
      \end{array}\right].
\]
Further,
\begin{equation}\label{dunimodular}
\begin{aligned}
   U_{D}\left(\lambda)^{-1} (e_{k}\otimes I_{r}\right) 
   &=\begin{cases}
       e_{k}\otimes I_{r} \,\,\,\,\ & \text{if $c_{0}(\mathbold{\beta})>0$}\\
       e_{k-i_{0}(\mathbold{\beta})}\otimes I_{r} & \text{if $c_{0}(\mathbold{\beta})=0$}
   \end{cases}\\
   \left(e^{T}_{k}\otimes I_{r}\right) V_{D}^{-1}(\lambda) &= e^{T}_{k-c_{0}(\mathbold{\beta})}\otimes I_{r}.
\end{aligned}
\end{equation}

\begin{theorem}\label{fiedlerproof}
Let $\mathbold{\alpha}$ be a permutation of $\{0,1,\dots,m-1\}$ and $T_{A}(\lambda)=\lambda M^{A}_{-m}-M^{A}_{\mathbold{\alpha}}$ be the Fiedler pencil of $A(\lambda)$ corresponding to $\mathbold{\alpha}$. Let $\mathbold{\beta}$ be a permutation of $\{0,1,\ldots,k-1\}$ and $T_{D}(\lambda)=\lambda N^{D}_{-k}-N^{D}_{\mathbold{\beta}}$ be the Fiedler pencil of $D(\lambda)$ corresponding to $\mathbold{\beta}$. Then, the pencil $\mathbb{T}(\lambda)$ defined as,
\[
\mathbb{T}(\lambda)
 =\left[\begin{array}{c|c}
      T_{A}(\lambda) & e_{m-i_{0}(\mathbold{\alpha})}e^{T}_{k-c_{0}(\mathbold{\beta})}\otimes -B\\ \hline
      e_{k-i_{0}(\mathbold{\beta})}e^{T}_{m-c_{0}(\mathbold{\alpha})}\otimes C & T_{D}(\lambda)
  \end{array}\right]
\]
is a Rosenbrock linearization of $G(\lambda).$
\end{theorem}

\begin{proof}
Define $\mathbb{U}(\lambda)=\left[\begin{array}{c|c} U_{A}(\lambda) & \\ \hline  & \left(J\otimes I_{r}\right)U_{D}(\lambda)\end{array}\right]$ and $\mathbb{V}(\lambda)=\left[\begin{array}{c|c} V_{A}(\lambda) & \\\hline  & V_{D}(\lambda)\left(J\otimes I_{r}\right)\end{array}\right]$ where $U_{i}(\lambda)$ and $V_{i}(\lambda)$ for $i\in\{A,D\}$ are the matrices defined in Lemma (\ref{unimodforfiedler}) and $J=\left[\begin{array}{ccc} 0 & \cdots & 1 \\ \vdots & \iddots & \vdots \\ 1 & \cdots & 0\end{array}\right]$. Then,
\begin{equation}\label{proofstatement01}
\begin{aligned}
\mathbb{U}(\lambda)\mathbb{T}(\lambda)\mathbb{V}(\lambda)=
&\left[\begin{array}{c|c}
    U_{A}(\lambda)T_{A}(\lambda)V_{A}(\lambda) & X_{12} \\ \hline 
    X_{21} & \left(J\otimes I_{r}\right)U_{D}(\lambda)T_{D}(\lambda)V_{D}(\lambda)\left(J\otimes I_{r}\right)
   \end{array}\right] \\
   \mbox{where } &X_{12}=U_{A}(\lambda)\left(e_{m-i_{0}(\mathbold{\alpha})}\otimes I_{n}\right)(-B)\left(e^{T}_{k-c_{0}(\mathbold{\beta})}\otimes I_{r}\right)V_{D}(\lambda)\left(J\otimes I_{r}\right) \mbox{ and}\\
   & X_{21}=\left(J\otimes I_{r}\right)U_{D}(\lambda)\left(e_{k-i_{0}(\mathbold{\beta})}\otimes I_{r}\right) C \left(e^{T}_{m-c_{0}(\mathbold{\alpha})}\otimes I_{n}\right)V_{A}(\lambda).
\end{aligned}
\end{equation}

\noindent \textbf{Case I:} Let us suppose that $\mathbold{\alpha}$ and $\mathbold{\beta}$ both have conseqution at $0$. Then, $c_{0}(\mathbold{\alpha})>0,$ $i_{0}(\mathbold{\alpha})=0$ and $c_{0}(\mathbold{\beta})>0,$ $i_{0}(\mathbold{\beta})=0$. By virtue of Equations (\ref{aunimodular}) and (\ref{dunimodular}), Equation (\ref{proofstatement01}) reduces to,
\begin{align*}
\mathbb{U}(\lambda)\mathbb{T}(\lambda)\mathbb{V}(\lambda)
&= \left[\begin{array}{c|c}
    U_{A}(\lambda)T_{A}(\lambda)V_{A}(\lambda) & X_{12} \\ \hline 
    X_{21} & \left(J\otimes I_{r}\right)U_{D}(\lambda)T_{D}(\lambda)V_{D}(\lambda)\left(J\otimes I_{r}\right)
   \end{array}\right] \\
   &\mbox{where } X_{12}=  U_{A}(\lambda)\left(e_{m}\otimes I_{n}\right)(-B)\left(e^{T}_{k-c_{0}(\mathbold{\beta})}\otimes I_{r}\right)V_{D}(\lambda)\left(J\otimes I_{r}\right) \mbox{ and}\\
   &{\hspace{1.1cm}} X_{21}=\left(J\otimes I_{r}\right)U_{D}(\lambda)\left(e_{k}\otimes I_{r}\right) C \left(e^{T}_{m-c_{0}(\mathbold{\alpha})}\otimes I_{n}\right)V_{A}(\lambda)\\
&=\left[\begin{array}{c|c}
    U_{A}(\lambda)T_{A}(\lambda)V_{A}(\lambda) & \left(e_{m}\otimes I_{n}\right)(-B)\left(e^{T}_{k}\otimes I_{r}\right)\left(J\otimes I_{r}\right) \\ \hline 
    \left(J\otimes I_{r}\right)\left(e_{k}\otimes I_{r}\right) C \left(e^{T}_{m}\otimes I_{n}\right) & \left(J\otimes I_{r}\right)U_{D}(\lambda)T_{D}(\lambda)V_{D}(\lambda)\left(J\otimes I_{r}\right)
   \end{array}\right]
\end{align*}
\begin{align*}
&\hspace{-3cm}=\left[\begin{array}{cc|cc}
          I_{(m-1)n}  & 0 & 0 & 0 \\
          0 & A(\lambda) & -B & 0 \\ \hline 
          0 & C & D(\lambda) & 0 \\ 
          0 & 0 & 0 & I_{(k-1)r}
    \end{array}\right].
\end{align*}

\noindent \textbf{Case II:} Let us suppose that $\mathbold{\alpha}$ has inversion at $0$ and $\mathbold{\beta}$ has conseqution at $0$. Then $c_{0}(\mathbold{\alpha})=i_{0}(\mathbold{\beta})=0$, $i_{0}(\mathbold{\alpha})>0$, $c_{0}(\mathbold{\beta})>0$ . Using Equations (\ref{aunimodular}) and (\ref{dunimodular}), Equation (\ref{proofstatement01}) reduces to,
\[
\mathbb{U}(\lambda)\mathbb{T}(\lambda)\mathbb{V}(\lambda)
= \left[\begin{array}{c|c}
    U_{A}(\lambda)T_{A}(\lambda)V_{A}(\lambda) & X_{12} \\ \hline 
    X_{21} & \left(J\otimes I_{r}\right)U_{D}(\lambda)T_{D}(\lambda)V_{D}(\lambda)\left(J\otimes I_{r}\right)
   \end{array}\right]
\]
\begin{align*}
   &\mbox{where} X_{12}=U_{A}(\lambda)\left(e_{m-i_{0}(\mathbold{\alpha})}\otimes I_{n}\right)(-B)\left(e^{T}_{k-c_{0}(\mathbold{\beta})}\otimes I_{r}\right)V_{D}(\lambda)\left(J\otimes I_{r}\right) \mbox{ and}\\
   &\hspace{1.1cm} X_{21}=\left(J\otimes I_{r}\right)U_{D}(\lambda)\left(e_{k}\otimes I_{r}\right) C \left(e^{T}_{m}\otimes I_{n}\right)V_{A}(\lambda)\\
&=\left[\begin{array}{c|c}
    U_{A}(\lambda)T_{A}(\lambda)V_{A}(\lambda) & \left(e_{m}\otimes I_{n}\right)(-B)\left(e^{T}_{k}\otimes I_{r}\right)\left(J\otimes I_{r}\right) \\ \hline 
    \left(J\otimes I_{r}\right)\left(e_{k}\otimes I_{r}\right) C \left(e^{T}_{m}\otimes I_{n}\right) & \left(J\otimes I_{r}\right)U_{D}(\lambda)T_{D}(\lambda)V_{D}(\lambda)\left(J\otimes I_{r}\right)
   \end{array}\right] \\
&=\left[\begin{array}{cc|cc}
          I_{(m-1)n}  & 0 & 0 & 0 \\
          0 & A(\lambda) & -B & 0 \\ \hline 
          0 & C & D(\lambda) & 0 \\ 
          0 & 0 & 0 & I_{(k-1)r}
    \end{array}\right].
\end{align*}

\noindent \textbf{Case III:} Let us suppose that $\mathbold{\alpha}$ and $\mathbold{\beta}$ both have inversion at $0$. Then, $c_{0}(\mathbold{\alpha})=0,$ $i_{0}(\mathbold{\alpha})> 0$ and $c_{0}(\mathbold{\beta})=0,$ $i_{0}(\mathbold{\beta})> 0$. Using Equations (\ref{aunimodular}) and (\ref{dunimodular}), Equation (\ref{proofstatement01}) reduces to,
\begin{align*}
\mathbb{U}(\lambda)\mathbb{T}(\lambda)\mathbb{V}(\lambda)
&=\left[\begin{array}{c|c}
    U_{A}(\lambda)T_{A}(\lambda)V_{A}(\lambda) & X_{12}\\ \hline 
    X_{21} & \left(J\otimes I_{r}\right)U_{D}(\lambda)T_{D}(\lambda)V_{D}(\lambda)\left(J\otimes I_{r}\right)
   \end{array}\right]\\
   &\mbox{where } X_{12}=U_{A}(\lambda)\left(e_{m-i_{0}(\mathbold{\alpha})}\otimes I_{n}\right)(-B)\left(e^{T}_{k}\otimes I_{r}\right)V_{D}(\lambda)\left(J\otimes I_{r}\right) \mbox{ and}\\
   &\hspace{1.1cm} X_{21}= \left(J\otimes I_{r}\right)U_{D}(\lambda)\left(e_{k-i_{0}(\mathbold{\beta})}\otimes I_{r}\right) C \left(e^{T}_{m}\otimes I_{n}\right)V_{A}(\lambda)\\
&=\left[\begin{array}{c|c}
    U_{A}(\lambda)T_{A}(\lambda)V_{A}(\lambda) & \left(e_{m}\otimes I_{n}\right)(-B)\left(e^{T}_{k}\otimes I_{r}\right)\left(J\otimes I_{r}\right) \\ \hline 
    \left(J\otimes I_{r}\right)\left(e_{k}\otimes I_{r}\right) C \left(e^{T}_{m}\otimes I_{n}\right) & \left(J\otimes I_{r}\right)U_{D}(\lambda)T_{D}(\lambda)V_{D}(\lambda)\left(J\otimes I_{r}\right)
   \end{array}\right]\\
&=\left[\begin{array}{cc|cc}
          I_{(m-1)n}  & 0 & 0 & 0 \\
          0 & A(\lambda) & -B & 0 \\ \hline 
          0 & C & D(\lambda) & 0 \\ 
          0 & 0 & 0 & I_{(k-1)r}
    \end{array}\right].
\end{align*}
Therefore, in all the cases, we see that $\mathbb{U}(\lambda)\mathbb{T}(\lambda)\mathbb{V}(\lambda)=\left[\begin{array}{c|c|c} I_{(m-1)n} &  & \\\hline  & \mathcal{S}(\lambda) & \\\hline & & I_{(k-1)r}\end{array}\right]$. This shows that $\mathbb{T}(\lambda)$ is a Rosenbrock linearization of $G(\lambda)$.
\end{proof}

\begin{lemma}\label{fiedlerdiag}
    Let $\mathbb{T}(\lambda)$ be the pencil defined in Theorem~\ref{fiedlerproof}. For nonsingular matrices $\mathcal{X}_{A},\mathcal{Y}_{A}\in\mathbb{C}^{mn\times mn}$ and $\mathcal{X}_{D},\mathcal{Y}_{D}\in\mathbb{C}^{rk\times rk}$ let $\mathbb{L}(\lambda)=\mathrm{diag}\left(\mathcal{X}_{A},\mathcal{X}_{D}\right)\mathbb{T}(\lambda)\mathrm{diag}\left(\mathcal{Y}_{A},\mathcal{Y}_{D}\right)$. Then, $\mathbb{L}(\lambda)$ is a Rosenbrock linearization of $G(\lambda)$.
\end{lemma}

\begin{proof}
    Since $\mathbb{T}(\lambda)$ is as defined in Theorem~\ref{fiedlerproof}, there exist unimodular matrix polynomials $U_{A}(\lambda)$, $U_{D}(\lambda)$, $V_{A}(\lambda)$, and $V_{D}(\lambda)$ such that,
    \begin{align*}
        \mathrm{diag}\left(I_{(m-1)n},\mathcal{S}(\lambda),I_{(k-1)r}\right)&\\
        &\hspace{-2cm}=\mathrm{diag}\left(U_{A}(\lambda),U_{D}(\lambda)\right)\mathbb{T}(\lambda)\mathrm{diag}\left(V_{A}(\lambda),V_{D}(\lambda)\right)\\
        &\hspace{-2cm}=\mathrm{diag}\left(U_{A}(\lambda)\mathcal{X}_{A}^{-1},U_{D}(\lambda)\mathcal{X}_{D}^{-1}\right)\mathbb{L}(\lambda)\mathrm{diag}\left(V_{A}(\lambda)\mathcal{Y}_{A}^{-1},V_{D}(\lambda)\mathcal{Y}_{D}^{-1}\right)
    \end{align*}
  Since $\mathcal{X}_{A}$, $\mathcal{X}_{D}$, $\mathcal{Y}_{A}$, and $\mathcal{Y}_{D}$ are nonsingular, the matrices $\mathrm{diag}\left(U_{A}(\lambda)\mathcal{X}_{A}^{-1},U_{D}(\lambda)\mathcal{X}_{D}^{-1}\right)$ and $\mathrm{diag}\left(V_{A}(\lambda)\mathcal{Y}_{A}^{-1},V_{D}(\lambda)\mathcal{Y}_{D}^{-1}\right)$ and the proof follows.
\end{proof}

In the following lemmas, we recall some important characteristics of GFPRs of a polynomial matrix.

\begin{lemma}[\cite{DA2019Automatic}]\label{1221blocks}
  Let $L_{A}(\lambda)=M_{(\mathbold{\tau}_{1},\mathbold{\sigma}_{1})}(Y^{A}_{1},X^{A}_{1})\left(\lambda M^{A}_{\mathbold{\tau}}-M^{A}_{\mathbold{\sigma}}\right)M_{(\mathbold{\sigma}_{2},\mathbold{\tau}_{2})}(X^{A}_{2},Y^{A}_{2})$ be the GFPR of $A(\lambda)$. If $\mathbold{\sigma}=(\theta_{1},0,\theta_{2})$, then $M^{A}_{\theta_{1}}(e_{m}\otimes I_{n})=e_{m-i_{0}(\mathbold{\sigma})}\otimes I_{n}$ and $(e^{T}_{m}\otimes I_{n})M^{A}_{\theta_{2}}=e^{T}_{m-c_{0}(\mathbold{\sigma})}\otimes I_{n}$.   
\end{lemma}

Analogously for $D(\lambda)$, if $L_{D}(\lambda)=N_{(\mathbold{\delta}_{1},\mathbold{\gamma}_{1})}(Y^{D}_{1},X^{D}_{1})\left(\lambda N^{D}_{\mathbold{\delta}}-N^{D}_{\mathbold{\gamma}}\right)N_{(\mathbold{\gamma}_{2},\mathbold{\delta}_{2})}(X^{D}_{2},Y^{D}_{2})$ be the GFPR of $D(\lambda)$ and if $\mathbold{\gamma}=(\pi_{1},0,\pi_{2})$, then $N^{D}_{\pi_{1}}(e_{k}\otimes I_{r})=e_{k-i_{0}(\mathbold{\gamma})}\otimes I_{r}$ and $(e^{T}_{k}\otimes I_{r})N^{D}_{\pi_{2}}=e^{T}_{k-c_{0}(\mathbold{\gamma})}\otimes I_{r}$.

\begin{lemma}[\cite{DA2019Automatic}]\label{2112blocks}
  Let $L_{A}(\lambda)=M_{(\mathbold{\tau}_{1},\mathbold{\sigma}_{1})}(Y^{A}_{1},X^{A}_{1})\left(\lambda M^{A}_{\mathbold{\tau}}-M^{A}_{\mathbold{\sigma}}\right)M_{(\mathbold{\sigma}_{2},\mathbold{\tau}_{2})}(X^{A}_{2},Y^{A}_{2})$ be the GFPR of $A(\lambda)$. Then 
  \begin{align*}
  \left( e^{T}_{m-c_{0}(\mathbold{\sigma})}\otimes I_{n}\right)M_{(\mathbold{\sigma}_{2},\mathbold{\tau}_{2})}(X^{A}_{2},Y^{A}_{2})&=e^{T}_{m-c_{0}(\mathbold{\sigma},\mathbold{\sigma}_{2})}\otimes I_{n} \mbox{ and}\\
  M_{(\mathbold{\tau}_{1},\mathbold{\sigma}_{1})}(Y^{A}_{1},X^{A}_{1})(e_{m-i_{0}(\mathbold{\sigma}}\otimes I_{n})&=e_{m-i_{0}(\mathbold{\sigma}_{1},\mathbold{\sigma})}\otimes I_{n}.
  \end{align*} 
\end{lemma}

Similarly, if $L_{D}(\lambda)=N_{(\mathbold{\delta}_{1},\mathbold{\gamma}_{1})}(Y^{D}_{1},X^{D}_{1})\left(\lambda N^{D}_{\mathbold{\delta}}-N^{D}_{\mathbold{\gamma}}\right)N_{(\mathbold{\gamma}_{2},\mathbold{\delta}_{2})}(X^{D}_{2},Y^{D}_{2})$ is the GFPR of $D(\lambda)$ then 
\begin{align*}
\left(e^{T}_{k-c_{0}(\mathbold{\gamma})}\otimes I_{r} \right)N_{(\mathbold{\gamma}_{2},\mathbold{\delta}_{2})}(X^{D}_{2},Y^{D}_{2})=e^{T}_{k-c_{0}(\mathbold{\gamma}_{1},\mathbold{\gamma})}\otimes I_{r}\mbox{ and}\\
N_{(\mathbold{\delta}_{1},\mathbold{\gamma}_{1})}(Y^{D}_{1},X^{D}_{1})\left(e_{k-i_{0}(\mathbold{\gamma})}\otimes I_{r}\right)=e_{k-i_{0}(\mathbold{\gamma}_{1},\mathbold{\gamma})}\otimes I_{r}.
\end{align*}

Now we come to one of our main results showing that for any rational matrix $G(\lambda)$, its GFPR is a linearization.

\begin{theorem}\label{gfprproof}
    Let
    \begin{equation*}
     \mathbb{L}(\lambda)=
        \left[\begin{array}{c|c}
            L_{A}(\lambda) & e_{m-i_{0}(\mathbold{\sigma}_{1},\mathbold{\sigma})}e^{T}_{c_{0}(\mathbold{\gamma},\mathbold{\gamma}_{2})}\otimes -B \\ \hline
            e_{k-i_{0}(\mathbold{\gamma}_{1},\mathbold{\gamma})}e^{T}_{m-c_{0}(\mathbold{\sigma},\mathbold{\sigma}_{2})}\otimes C & L_{D}(\lambda)
        \end{array}\right]
    \end{equation*}
    be the GFPR of $G(\lambda)$ as stated in Definition (\ref{gfprdefinition}), where all the matrix assignments $X^{i}_{j}$ and $Y^{i}_{j}$, $j=1,2$ and $i\in\{A,D\}$ are nonsingular. Then $\mathbb{L}(\lambda)$ is a Rosenbrock linearization of $G(\lambda)$.
\end{theorem}

\begin{proof}
    Let $\mathbold{\tau}=(\phi,-m,\psi)$ and define $\mathbold{\alpha}=\left(-rev(\phi),\mathbold{\sigma},-rev(\psi)\right)$ and $T_{A}(\lambda)= \lambda M^{A}_{-m}-M^{A}_{\mathbold{\alpha}}$. Next, let $\mathbold{\delta}=(\mu,-k,\omega)$ and define $\mathbold{\beta}=\left(-rev(\mu),\mathbold{\gamma},-rev(\omega)\right)$ and $T_{D}(\lambda)=\lambda N^{D}_{-k}-N^{D}_{\mathbold{\beta}}$. Then $\mathbold{\alpha}$ is a permutation of $\{0,1,\ldots,m-1\}$ and $T_{A}(\lambda)=\lambda M^{A}_{-m}-M^{A}_{\mathbold{\alpha}}$ is a Fiedler pencil of $A(\lambda)$ corresponding to $\mathbold{\alpha}$. Similarly, $\mathbold{\beta}$ is a permutation of $\{0,1,\ldots,k-1\}$ and $T_{D}(\lambda)=\lambda N^{D}_{-k}-N^{D}_{\mathbold{\beta}}$ is a Fiedler pencil of $D(\lambda)$ corresponding to $\mathbold{\beta}$. Thus the pencil $\mathbb{T}(\lambda)$ defined as,
\[
\mathbb{T}(\lambda)
 =\left[\begin{array}{c|c}
      T_{A}(\lambda) & e_{m-c_{0}(\mathbold{\alpha})}e^{T}_{k-i_{0}(\mathbold{\beta})}\otimes -B\\ \hline
      e_{k-c_{0}(\mathbold{\beta})}e^{T}_{m-i_{0}(\mathbold{\alpha})}\otimes C & T_{D}(\lambda)
  \end{array}\right]
\]
is same as the pencil defined in Theorem (\ref{fiedlerproof}). Then, for the GFPR $L_{A}(\lambda)$ of $A(\lambda)$, we have,
\begin{align*}
    L_{A}(\lambda)
         &=M_{(\mathbold{\tau}_{1},\mathbold{\sigma}_{1})}(Y^{A}_{1},X^{A}_{1})\left(\lambda M^{A}_{\mathbold{\tau}}-M^{A}_{\mathbold{\sigma}}\right)M_{(\mathbold{\sigma}_{2},\mathbold{\tau}_{2})}(X^{A}_{2},Y^{A}_{2})\\
         &=M_{(\mathbold{\tau}_{1},\mathbold{\sigma}_{1})}(Y^{A}_{1},X^{A}_{1})\left(\lambda M^{A}_{(\phi,-m,\psi)}-M^{A}_{\phi}M^{A}_{-rev(\phi)}M^{A}_{\mathbold{\sigma}}M^{A}_{-rev(\psi)}M^{A}_{\psi}\right)\\
         &\hspace{9cm} \times M_{(\mathbold{\sigma}_{2},\mathbold{\tau}_{2})}(X^{A}_{2},Y^{A}_{2})\\
         &=M_{(\mathbold{\tau}_{1},\mathbold{\sigma}_{1})}(Y^{A}_{1},X^{A}_{1})\left(\lambda M^{A}_{\phi}M^{A}_{-m}M^{A}_{\psi}-M^{A}_{\phi}M^{A}_{(-rev(\phi),\mathbold{\sigma},-rev(\psi))}M^{A}_{\psi}\right)\\
         &\hspace{9cm}\times M_{(\mathbold{\sigma}_{2},\mathbold{\tau}_{2})}(X^{A}_{2},Y^{A}_{2})\\
         &=M_{(\mathbold{\tau}_{1},\mathbold{\sigma}_{1})}(Y^{A}_{1},X^{A}_{1})M^{A}_{\phi}\left(\lambda M^{A}_{-m}-M^{A}_{(-rev(\phi),\mathbold{\sigma},-rev(\psi))}\right)M^{A}_{\psi}M_{(\mathbold{\sigma}_{2},\mathbold{\tau}_{2})}(X^{A}_{2},Y^{A}_{2})\\
         &=\mathcal{X}_{A} \left(\lambda M^{A}_{-m}-M^{A}_{(-rev(\phi),\mathbold{\sigma},-rev(\psi))}\right)\mathcal{Y}_{A}\\
         &=\mathcal{X}_{A} \left(\lambda M^{A}_{-m}-M^{A}_{\mathbold{\alpha}}\right)\mathcal{Y}_{A}\\
         &=\mathcal{X}_{A} T_{A}(\lambda)\mathcal{Y}_{A},
\end{align*}
where $\mathcal{X}_{A}=M_{(\mathbold{\tau}_{1},\mathbold{\sigma}_{1})}(Y^{A}_{1},X^{A}_{1})M^{A}_{\phi}$, $\mathcal{Y}_{A}=M^{A}_{\psi}M_{(\mathbold{\sigma}_{2},\mathbold{\tau}_{2})}(X^{A}_{2},Y^{A}_{2})$. Since the matrix assignments $X^{A}_{1}, X^{A}_{2}$ and $Y^{A}_{1},Y^{A}_{2}$ are non-singular, the matrices $M_{(\mathbold{\tau}_{1},\mathbold{\sigma}_{1})}(Y^{A}_{1},X^{A}_{1})$ and $M_{(\mathbold{\sigma}_{2},\mathbold{\tau}_{2})}(X^{A}_{2},Y^{A}_{2})$ are nonsingular and consuquently $\mathcal{X}_{A}$ and $\mathcal{Y}_{A}$ are nonsingular.

Following the same line of reasoning for the GFPR $L_{D}(\lambda)$ of $D(\lambda)$, we have, 
\begin{align*}
L_{D}(\lambda)
       =\mathcal{X}_{D} \left(\lambda N^{A}_{-k}-N^{D}_{\mathbold{\beta}}\right)\mathcal{Y}_{D} 
       =\mathcal{X}_{D} T_{D}(\lambda)\mathcal{Y}_{D},
\end{align*}
where $\mathcal{X}_{D}=N_{(\mathbold{\delta}_{1},\mathbold{\gamma_{1}})}(Y^{D}_{1},X^{D}_{1})N^{D}_{\mu}$, $\mathcal{Y}_{D}=N^{D}_{\omega}N_{(\mathbold{\gamma}_{2},\mathbold{\delta}_{2})}(X^{D}_{2},Y^{D}_{2})$ are nonsingular and $\mathbold{\beta}$ is a permutation on $\{0:k-1\}$.

Next, focus on the $(1,2)$ block of $\mathbb{L}(\lambda)$ and using Lemmas (\ref{1221blocks}) and (\ref{2112blocks}) we simplify it as follows.

\noindent
\begin{align*}
    e_{m-i_{0}(\mathbold{\sigma}_{1},\mathbold{\sigma})}&e^{T}_{c_{0}(\mathbold{\gamma},\mathbold{\gamma}_{2})}\otimes -B\\
    &=\left(e_{m-i_{0}(\mathbold{\sigma}_{1},\mathbold{\sigma})}\otimes I_{n}\right)(-B)\left(e^{T}_{k-c_{0}(\mathbold{\gamma},\mathbold{\gamma}_{2})}\otimes I_{k}\right)\\
    &=M_{(\mathbold{\tau}_{1},\mathbold{\sigma}_{1})}(Y^{A}_{1},X^{A}_{1})\left(e_{m-i_{0}(\sigma)}\otimes I_{n}\right) (-B) \left(e^{T}_{k-c_{0}(\gamma)}\otimes I_{r}\right)N_{(\mathbold{\gamma}_{2},\mathbold{\delta}_{2})}(X^{D}_{2},Y^{D}_{2})\\
    &=M_{(\mathbold{\tau}_{1},\mathbold{\sigma}_{1})}(Y^{A}_{1},X^{A}_{1})M^{A}_{\phi}M^{A}_{-rev(\phi)}\left(e_{m-i_{0}(\mathbold{\sigma})}\otimes I_{n}\right) \\
    &\hspace{3cm}\times(-B) \left(e^{T}_{k-c_{0}(\mathbold{\gamma})}\otimes I_{r}\right)N^{D}_{-rev(\omega)}N^{D}_{\omega}N_{(\mathbold{\gamma}_{2},\mathbold{\delta}_{2})}(X^{D}_{2},Y^{D}_{2})\\
    &=\mathcal{X}_{A}M^{A}_{-rev(\phi)}\left(e_{m-i_{0}(h_{1},0,h_{2})}\otimes I_{n}\right) (-B) \left(e^{T}_{k-c_{0}(\ell_{1},0,\ell_{2})}\otimes I_{r}\right)N^{D}_{-rev(\omega)}\mathcal{Y}_{D}\\
    &=\mathcal{X}_{A}M^{A}_{-rev(\phi)}M_{h_{1}}\left(e_{m}\otimes I_{n}\right) (-B) \left(e^{T}_{k}\otimes I_{r}\right)N^{D}_{\ell_{2}}N^{D}_{-rev(\omega)}\mathcal{Y}_{D}\\
    &=\mathcal{X}_{A}M^{A}_{(-rev(\phi),h_{1})}\left(e_{m}\otimes I_{n}\right) (-B) \left(e^{T}_{k}\otimes I_{r}\right)N^{D}_{(\ell_{2},-rev(\omega))}\mathcal{Y}_{D}\\
    &=\mathcal{X}_{A}\left(e_{m-i_{0}(\mathbold{\alpha})}\otimes I_{n}\right) (-B) \left(e^{T}_{k-c_{0}(\mathbold{\beta})}\otimes I_{r}\right)\mathcal{Y}_{D}]\\
    &=\mathcal{X}_{A}\left(e_{m-i_{0}(\mathbold{\alpha})}e^{T}_{k-c_{0}(\mathbold{\beta})}\otimes -B\right)\mathcal{Y}_{D}.
\end{align*}
For the $(2,1)$ block, retracing the steps as above and using Lemmas (\ref{1221blocks}) and (\ref{2112blocks}), we can write it as,
\begin{align*}
    e_{k-i_{0}(\mathbold{\gamma}_{1},\mathbold{\gamma})}e^{T}_{m-c_{0}(\mathbold{\sigma},\mathbold{\sigma}_{2})}\otimes C &= \mathcal{X_{D}}\left(e_{k-c_{0}(\mathbold{\beta})}e^{T}_{m-i_{0}(\mathbold{\alpha})}\otimes C\right)\mathcal{Y}_{A}.
\end{align*}
Therefore, 
\begin{align*}
    \mathbb{L}(\lambda)
    &=\left[\begin{array}{c|c}
            L_{A}(\lambda) & e_{m-i_{0}(\mathbold{\sigma}_{1},\mathbold{\sigma})}e^{T}_{c_{0}(\mathbold{\gamma},\mathbold{\gamma}_{2})}\otimes -B \\ \hline
            e_{k-i_{0}(\mathbold{\gamma}_{1},\mathbold{\gamma})}e^{T}_{m-c_{0}(\mathbold{\sigma},\mathbold{\sigma}_{2})}\otimes C & L_{D}(\lambda)
        \end{array}\right]\\
    &=\left[\begin{array}{c|c}
            \mathcal{X}_{A}T_{A}(\lambda)\mathcal{Y}_{A} & \mathcal{X}_{A}\left(e_{m-i_{0}(\mathbold{\alpha})}e^{T}_{k-c_{0}(\mathbold{\beta})}\otimes -B\right)\mathcal{Y}_{D} \\ \hline
            \mathcal{X_{D}}\left(e_{k-c_{0}(\mathbold{\beta})}e^{T}_{m-i_{0}(\mathbold{\alpha})}\otimes C\right)\mathcal{Y}_{A} & \mathcal{X}_{D}T_{D}(\lambda)\mathcal{Y}_{D}
        \end{array}\right] \\
    &=\left[\begin{array}{c|c}
            \mathcal{X}_{A} &  \\\hline
                  & \mathcal{X_{D}}
      \end{array}\right]
      \left[\begin{array}{c|c}
            T_{A}(\lambda) & \left(e_{m-i_{0}(\mathbold{\alpha})}e^{T}_{k-c_{0}(\mathbold{\beta})}\otimes -B\right) \\ \hline
            \left(e_{k-c_{0}(\mathbold{\beta})}e^{T}_{m-i_{0}(\mathbold{\alpha})}\otimes C\right)& T_{D}(\lambda)
      \end{array}\right] 
      \left[\begin{array}{c|c}
            \mathcal{Y}_{A} &  \\\hline
                  & \mathcal{Y_{D}}
      \end{array}\right]\\
    &=\mathrm{diag}\left(\mathcal{X}_{A},\mathcal{X}_{D}\right)\mathbb{T}(\lambda)\mathrm{diag}\left(\mathcal{Y}_{A},\mathcal{Y}_{D}\right).
\end{align*}
Hence, by Lemma (\ref{fiedlerdiag}), it follows that $\mathbb{L}(\lambda)$ is a Rosenbrock linearization of $G(\lambda)$.
\end{proof}

\section{Structured Lineraizations}
In this section we consider rational matrices with structures and construct linearizations which preserve these structures. Our focus will be on those rational matrices that are symmetric, skew-symmetric, $T$-even and $T$-odd. We show that that the falimy of GFPRs generate a large variety of linearizations of $G(\lambda)$ which preserve structures. Let us recall that our rational matrix has the form,
\begin{equation}\label{tfunction01}
G(\lambda) = C A(\lambda)^{-1} B + D(\lambda) \in \mathbb{C}(\lambda)^{r \times r}. 
\end{equation}
which we consider as a transfer function of an LTI system $\Sigma$ in (\ref{my_lti_chap05}) with the system matrix given by,
\begin{equation}\label{sysmatrix01}
\mathcal{S}(\lambda) = \left[\begin{array}{c|c}
                            A(\lambda) & -B \\ \hline
                            C & D (\lambda) \\
                       \end{array}\right] 
                       \in \mathbb{C}[\lambda]^{(n+r)\times (n+r)} 
\end{equation}
where $A(\lambda) = \sum_{j=0}^{m}\lambda^{j}A_j \in \mathbb{C}[\lambda]^{n \times n}$ is regular with degree $m$ and $ D(\lambda) = \sum_{j=0}^{k}\lambda^{j}D_j \in \mathbb{C}[\lambda]^{r \times r}$ is of degree $k$,  $ C \in \mathbb{C}^{r \times n}, B \in \mathbb{C}^{n \times r}$.

\subsection{Symmetric GFPRs}
A rational matrix $G(\lambda)$ is symmetric if $G(\lambda)^{T}=G(\lambda)$. A rational matrix function has a symmetric realization if $G(\lambda)=D(\lambda)+CA(\lambda)^{-1}B$ with $A(\lambda)$ and $D(\lambda)$ as symmetric and $C=B^{T}$. A system matrix $\mathcal{S}(\lambda)=\left[\begin{array}{c|c} A(\lambda) & B \\\hline C & D(\lambda)\end{array}\right]$ is said to be symmetric if $A(\lambda)$ and $D(\lambda)$ are symmetric and $C=B^{T}$. It is easy to check that $G(\lambda)=D(\lambda)+CA(\lambda)^{-1}B$ is a symmetric realization of $G(\lambda)$ if and only if $\mathcal{S}(\lambda)=\left[\begin{array}{c|c} A(\lambda) & B \\\hline C & D(\lambda)\end{array}\right]$ is a symmetric system matrix of $G(\lambda)$. In this section, we have constructed symmetric linarization of $G(\lambda)$ having a symmetric realization.

The block transpose of a $p\times q$ block matrix $\mathcal{M}=[M_{ij}]$ is a $q\times p$ block matrix $\mathcal{M}^{\mathcal{C}}$ where $\mathcal{M}^{\mathcal{B}}=[M_{ji}]$. A block matrix $\mathcal{M}$ is said to be block symmetric provided that $\mathcal{M}^{\mathcal{B}}=\mathcal{M}$, see \cite{TDM2010}. Recall that if $\mathbb{S}$ be a $(mn+rk)\times (mn+rk)$ system matrix given by 
\[
\mathbb{S} = \left[\begin{array}{c|c}
                   \mathcal{A} & uv^{T}\otimes B \\\hline
                   sz^{T}\otimes C & \mathcal{D} \\
               \end{array}\right],
\]
where $u,z\in\mathbb{C}^{m}$, $v,s\in\mathbb{C}^{k}$, $B\in \mathbb{C}^{n\times r}$, $C\in \mathbb{C}^{r\times n}$, ${\mathcal A} = [{\mathcal A}_{ij}]$ is an $m \times m$ block matrix with ${\mathcal A}_{ij} \in {\mathbb C}^{n \times n}$, and ${\mathcal D}=[\mathcal{D}_{ij}]$ is a $k \times k$ block matrix with $\mathcal{D}_{ij} \in {\mathbb C}^{r \times r}$. The Rosenbrock block transpose of $\mathbb{S}$, denoted by $\mathbb{S}^{\mathbb{B}}$ is defined by
\[\mathbb{S}^{\mathbb{B}} 
      = \left[\begin{array}{c|c}
              \mathcal{A}^{\mathcal{B}} & zs^{T} \otimes B \\\hline
              vu^{T}  \otimes C & \mathcal{D}^{\mathcal{B}} \\
         \end{array}\right],
\]
where $\mathcal{A}^{\mathcal{B}}$ and $\mathcal{D}^{\mathcal{B}}$ denote  the block transpose of $\mathcal{A}$.

Observe that $\mathbb{S}$ is block-symmetric if and only if $\mathcal{A}$ and $\mathcal{B}$ are block-symmetric and $u=z$ and $v=s$. Now we present some important definitions.

\begin{definition}[\cite{BDFR2015}]\label{indtupleandsymmcomp}
 \begin{enumerate}[label=(\alph*)]
     \item  Let $r\in\mathbb{Z}$ and $r\geq 0$. If $\mathbold{\gamma}$ is a permutation of $\{0,1,\ldots,r\}$ and
     \begin{equation}\label{admtuple}
         csf(\mathbold{\gamma})=\left(r-1:r,r-3:r-2,\ldots,q+1,0:q\right)
     \end{equation}
     for some $0\leq q\leq r$ then $ \mathbold{\gamma}$ is said to be an admissible tuple of $\{0,1,\ldots,r\}$. The integer $q$ is called the index of $\mathbold{r}$ and is denoted by $\mathrm{Ind}(\mathbold{r})$.

     \item Let $r\in\mathbb{Z}$ and $r\geq 0$. Consider an admissible tuple $\mathbold{\gamma}$ of $\{0,1,\ldots,r\}$ with index $q$. Then the symmetric complement $\mathbold{c}_{\gamma}$ of $\mathbold{\gamma}$ is defined as
     \[
     \mathbold{ c}_{\gamma}=
     \begin{cases}
         \left(r-1,r-3,\ldots,q+3,q+1,(0:q)_{rev_{c}}\right) \,\,\,\,\ &\text{for $q\geq1$,}\\
        \left(r-1,r-3,\ldots,1\right) & \text{if $q=0$ and $r>0$,}\\
        \phi & \text{if $r=0,$}
     \end{cases}
     \]
     where $(0:q)_{rev_{c}}=\left(0:q-1,0:q-2,\ldots,0:1,0\right)$.
 \end{enumerate}    
\end{definition}

Throughout, we work with admissible tuples as in (\ref{admtuple}). It is known that for every $r\geq0$ there is a unique admissible tuple with index $0$ and $1$\cite{BDFR2015}.

\begin{definition}[\cite{DA2022}]
    A simple admissible tuple is an admissible tuple with index either $0$ or $1$.
\end{definition}

It is to be noted that for $r\in\mathbb{Z}$ with $r\geq0$ and a simple admissible tuple $\mathbold{\gamma}$ of $\{0,1,\ldots,r\}$  $\mathrm{Ind}(\mathbold{\gamma})=0$ when $r$ is even and $\mathrm{Ind}(\mathbold{\gamma})=1$ when $r$ is odd.

\begin{remark}
{\rm
    For an admissible tuple $\mathbold{\gamma}$ of $\{0,1,\ldots,r\}$ with symmetric complement $\mathbold{ c}_{\gamma}$, Definition \ref{indtupleandsymmcomp} implies that $0\in\mathbold{c}_{\gamma}$ whenever $\mathrm{Ind}(\mathbold{\gamma})\geq1$. 
    
    If $\mathbold{\gamma}$ is a simple admissible tuple of $\{0,1,\ldots,r\}$ then $0\in\mathbold{c}_{\gamma}$ when $r$ is odd and $0\notin\mathbold{ c}_{\gamma}$ when $r$ even. }
\end{remark}

\begin{definition}[\cite{BDFR2015}]
    Let $r\in\mathbb{Z}$ and $r\geq0$. If an index tuple $\mathbold{\gamma}$ has the form
    \[\left(s_{1}:r-2,s_{2}:r-4,\ldots,s_{\left\lfloor \frac{r}{2}\right\rfloor}:r-2\left\lfloor\frac{r}{2}\right\rfloor\right)
    \]
    where $s_{j}\geq0$ for $j=1,2,\ldots,\left\lfloor\frac{h}{2}\right\rfloor$, where $\lfloor \cdot\rfloor$ stands for the greatest integer function, then $\mathbold{\gamma}$ is said to be in a canonical form for $r$.
\end{definition}

If $r=0$ or $1$ then the index tuple in canonical form for $r$ is empty.


\begin{theorem}[\cite{BDFR2015}]\label{symmetricpolypencil}
    Let $h\in\{0,1,\ldots, m-1\}$. Consider the simple admissible tuples $\mathbold{ w}_{h}$ and $\mathbold{v}_{h}+m$ of $\{0,1,\ldots,h\}$ and $\{0,1,\ldots,m-h-1\}$, respectively. Consider the index tuples in canonical forms $\mathbold{ t}_{w_{h}}$ and $\mathbold{t}_{v_{h}}+m$ for $h$ and $m-h-1$, respectively. For $\mathbold{t}_{w_{h}}$ and $\mathbold{ t}_{v_{h}}$, let $\mathcal{X}$ and $\mathcal{Y}$ be nonsingular matrix assignments. Then
    \begin{equation}\label{polygfpr}
L(\lambda)=M_{(\mathbold{t}_{v_{h}},\mathbold{ t}_{w_{h}})}\left(\mathcal{Y},\mathcal{X}\right) \left(\lambda M^{A}_{\mathbold{ v}_{h}}-M^{A}_{\mathbold{ w}_{h}}\right) M_{(rev(\mathbold{t}_{ w_{h}}),rev(\mathbold{ t}_{ v_{h}}))}\left(rev(\mathcal{X}),rev(\mathcal{Y})\right),
    \end{equation}
    is a GFPR of $A(\lambda)$ which is block symmetric. Further, the matrix assignments $\mathcal{X}$ and $\mathcal{Y}$ contain only symmetric matrices, then $L(\lambda)$ is symmetric whenever $A(\lambda)$ is symmetric.
\end{theorem}

\begin{definition}\label{sysmatrixgfprdefn}
    Let $h\in\{0,1,\ldots, m-1\}$. Consider the simple admissible tuples $\mathbold{ w}_{h}$ and $\mathbold{v}_{h}+m$ of $\{0,1,\ldots,h\}$ and $\{0,1,\ldots,m-h-1\}$, respectively. Consider the index tuples in canonical forms $\mathbold{ t}_{w_{h}}$ and $\mathbold{t}_{v_{h}}+m$ for $h$ and $m-h-1$, respectively. Let $X_{A}$ and $Y_{A}$ be nonsingular matrix assignments for $\mathbold{ t}_{w_{h}}$ and $\mathbold{ t}_{ v_{h}}$, respectively. Let $\ell\in\{0,1,\ldots, k-1\}$. Consider the simple admissible tuples $\mathbold{ w}_{\ell}$ and $\mathbold{ v}_{\ell}+k$ of $\{0,1,\ldots,\ell\}$ and $\{0,1,\ldots,k-\ell-1\}$, respectively. For the index tuples in canonical forms $\mathbold{ t}_{w_{\ell}}$ and $\mathbold{ t}_{v_{\ell}}+k$ $\ell$ and $k-\ell-1$, respectively, let $X_{D}$ and $Y_{D}$ be nonsingular matrix assignments. We define a pencil $\mathbb{L}(\lambda)$ of $\mathcal{S}(\lambda)$ as
    \begin{equation}\label{sysmatrixgfpr}
    \mathbb{L}(\lambda)
          =\left[\begin{array}{c|c}
             L_{A}(\lambda) & e_{m-i_{0}(\mathbold{ t}_{w_{h}},\mathbold{ w}_{h})}e^{T}_{k-c_{0}(\mathbold{w}_{\ell},\mathbold{ c}_{w_{\ell}},rev(\mathbold{ t}_{ w_{\ell}}))}\otimes B \\ \hline
             e_{k-i_{0}(\mathbold{ t}_{ w_{\ell}},\mathbold{ w}_{\ell})}e^{T}_{m-c_{0}(\mathbold{ w}_{h},\mathbold{ c}_{ w_{h}},rev(\mathbold{ t}_{w_{h}}))}\otimes C & L_{D}(\lambda)
           \end{array}\right],
    \end{equation}
where 
\begin{align*}
L_{A}(\lambda)&=M_{(\mathbold{ t}_{ v_{h}},\mathbold{ t}_{w_{h}})}\left(Y_{A},X_{A}\right) \left(\lambda M^{A}_{\mathbold{ v}_{h}}-M^{A}_{\mathbold{ w}_{h}}\right)M^{A}_{(\mathbold{ c}_{ w_{h}},\mathbold{ c}_{v_{h}})} M_{(rev(\mathbold{ t}_{ w_{h}}),rev(\mathbold{ t}_{v_{h}}))}\left(rev(X_{A}),rev(Y_{A})\right)\\  L_{D}(\lambda)&=N_{(\mathbold{ t}_{v_{\ell}},\mathbold{ t}_{w_{\ell}})}\left(Y_{D},X_{D}\right) \left(\lambda N^{D}_{\mathbold{ v}_{\ell}}-N^{D}_{\mathbold{ w}_{\ell}}\right)N^{D}_{(\mathbold{ c}_{ w_{\ell}},\mathbold{ c}_{v_{\ell}})}N_{(rev(\mathbold{ t}_{ w_{\ell}}),rev(\mathbold{ t}_{v_{\ell}}))}\left(rev(X_{D}),rev(Y_{D})\right)
\end{align*}
are block symmetric GFPRs of $A(\lambda)$ and $D(\lambda)$ respectively.
\end{definition}

\begin{theorem}
    Let $\mathcal{S}(\lambda)$ be as in (\ref{sysmatrix01}). Let $h\in\{0,1,\ldots, m-1\}$ and $\ell\in\{0,1,\ldots k-1\}$ be even. Then $\mathbb{L}(\lambda)$ given in Definition (\ref{sysmatrixgfprdefn}) by Equation (\ref{sysmatrixgfpr}) is a block symmetric GFPR of $\mathcal{S}(\lambda)$.
    
    Further, if $m$ and $k$ both are odd then $\mathbb{L}(\lambda)$ is a Rosenbrock linearization of $\mathcal{S}(\lambda)$ and if $m$ (or $k$ or both) is (are) even then $\mathbb{L}(\lambda)$ is a Rosenbrock linearization of $\mathcal{S}(\lambda)$ if the leading coefficient(s) of $A(\lambda)$ (or $D(\lambda)$ or both) is (are) nonsingular. 
\end{theorem}

\begin{proof}
    By Theorem \ref{symmetricpolypencil}, $L_{A}(\lambda)$ and $L_{D}(\lambda)$ are block symmetric GFPRs of $A(\lambda)$ and $D(\lambda)$ respectively. Hence $\mathbb{L}(\lambda)$ is block symmetric if and only if $i_{0}(\mathbold{ t}_{w_{h}},\mathbold{ w}_{h})=c_{0}(\mathbold{ w}_{h},\mathbold{ c}_{w_{h}},rev(\mathbold{ t}_{w_{h}}))$ and $i_{0}(\mathbold{ t}_{w_{\ell}},\mathbold{ w}_{\ell})=c_{0}(\mathbold{ w}_{\ell},\mathbold{ c}_{w_{\ell}},rev(\mathbold{ t}_{w_{\ell}}))$. 

    \noindent \textbf{Case I:} Suppose $h=0$ and $l=0$. Then $\mathbold{ w}_{h}=(0)$, $c_{w_{h}}=\phi$, and, $rev(\mathbold{ t}_{w_{h}})=\phi$ (since $\mathbold{ t}_{w_{h}}=\phi$) and $\mathbold{ w}_{\ell}=(0)$, $c_{w_{\ell}}=\phi$, and, $rev(\mathbold{ t}_{w_{\ell}})=\phi$ (since $\mathbold{ t}_{w_{\ell}}=\phi$). Therefore, $i_{0}(\mathbold{ t}_{w_{h}},\mathbold{ w}_{h})=c_{0}(\mathbold{ w}_{h},\mathbold{ c}_{w_{h}},rev(\mathbold{ t}_{w_{h}}))=0$ and $i_{0}(\mathbold{ t}_{w_{\ell}},\mathbold{ w}_{\ell})=c_{0}(\mathbold{ w}_{\ell},\mathbold{ c}_{w_{\ell}},rev(\mathbold{ t}_{w_{\ell}}))=0$

    \noindent \textbf{Case II:} Suppose that $h>0$ and $l=0$. Since $h$ is even and $\mathbold{ w}_{h}$ is a simple admissible tuple of $\{0,1,\ldots,h\}$, $\mathbold{ w}_{h}=(h-1:h,h-3:h-2,\ldots,1:2,0)$ and $\mathbold{ c}_{w_{h}}=(h-1,h-3,\ldots,3,1)$. Thus, $c_{0}(\mathbold{ w}_{h},\mathbold{ c}_{w_{h}},rev(\mathbold{ t}_{w_{h}}))=2+c_{2}(rev(\mathbold{ t}_{w_{h}})$ and $i_{0}(\mathbold{ t}_{w_{h}},\mathbold{ w}_{h})=2+i_{2}(\mathbold{ t}_{w_{h}})$. Since for any index tuple $\alpha$ and an index $t$, $c_{t}(rev(\alpha))=i_{t}(\beta)$ we have $i_{0}(\mathbold{ t}_{w_{h}},\mathbold{ w}_{h})=c_{0}(\mathbold{ w}_{h},\mathbold{ c}_{w_{h}},rev(\mathbold{ t}_{w_{h}}))$. Also since $\ell=0$,  $i_{0}(\mathbold{ t}_{w_{\ell}},\mathbold{ w}_{\ell})=c_{0}(\mathbold{ w}_{\ell},\mathbold{ c}_{w_{\ell}},rev(\mathbold{ t}_{w_{\ell}}))=0$ follows from Case I.

    \noindent \textbf{Case III:} Suppose $h>0$ and $\ell>0$. Then the required equalities can be obtained by following the first part of Case II.
    
This shows that $\mathbb{L}(\lambda)$ is block symmetric.

Next since $h$ is even and $\mathbold{ w}_{h}$ is a simple admissible tuple of $\{0,1,\ldots,h\}$, $\mathrm{Ind}\mathbold{ w}_{h}=0$. Then $\mathbold{ c}_{w_{h}}=(h-1,h-3,\ldots,1)$. Therefore, $0\notin\mathbold{ c}_{w_{h}}$ and hence the matrix assignment for $\mathbold{ c}_{w_{h}}$ is nonsingular. Further the matrix assignmnents $X_{A}$ and $Y_{A}$ for $\mathbold{ t}_{w_{h}}$ and $\mathbold{ t}_{v_{h}}$ are nonsingular.

Similarly since $\ell$ is even and $\mathbold{ w}_{\ell}$ is a simple admissible tuple of $\{0,1,\ldots,\ell\}$, $\mathrm{Ind}(\mathbold{ w}_{\ell})=0$. Then $\mathbold{ c}_{w_{\ell}}=(\ell-1,\ell-3,\ldots,1)$. Therefore, $0\notin\mathbold{ c}_{w_{\ell}}$ and hence the matrix assignment for $\mathbold{ c}_{w_{\ell}}$ is nonsingular. Further the matrix assignments $X_{D}$ and $Y_{D}$ for $\mathbold{ t}_{w_{\ell}}$ and $\mathbold{ t}_{v_{\ell}}$ are nonsingular.

 Now take $\mathbold{\sigma}=\mathbold{ w}_{h}$, $\mathbold{\tau}=\mathbold{ v}_{h}$, $\mathbold{\sigma}_{1}=\mathbold{ t}_{w_{h}}$, $\mathbold{\sigma}_{2}=(\mathbold{ c}_{w_{h}},rev(\mathbold{ t}_{w_{h}}))$, $\mathbold{\tau}_{1}=\mathbold{ t}_{v_{h}}$, $\mathbold{\tau}_{2}=(\mathbold{ c}_{v_{h}},rev(\mathbold{ t}_{v_{h}}))$ and $\mathbold{\gamma}=\mathbold{ w}_{\ell}$, $\mathbold{\delta}=\mathbold{ v}_{\ell}$, $\mathbold{\gamma}_{1}=\mathbold{ t}_{w_{\ell}}$, $\mathbold{\gamma}_{2}=(\mathbold{ c}_{w_{\ell}},rev(\mathbold{ t}_{w_{\ell}}))$, $\mathbold{\delta}_{1}=\mathbold{ t}_{v_{\ell}}$, $\mathbold{\delta}_{2}=(\mathbold{ c}_{v_{\ell}},rev(\mathbold{ t}_{v_{\ell}}))$. Then, $\mathbb{L}(\lambda)$ is a Rosenbrock linearization of $\mathcal{S}(\lambda)$ if matrix assignments for  $\mathbold{ c}_{v_{h}}$ and $\mathbold{ c}_{v_{\ell}}$ are nonsingular.

 Suppose $m$ is odd, then $m-h-1$ is even as $h$ is even. So, $0\notin\mathbold{ c}_{v_{h}}+m\implies -m\notin\mathbold{ c}_{v_{h}}$ which shows that the matrix assignment for $\mathbold{ c}_{v_{h}}$ is nonsingular. Similarly, if $k$ is odd, then $k-\ell-1$ is even as $\ell$ is even. So, $0\notin\mathbold{ c}_{v_{\ell}}+k\implies -k\notin\mathbold{ c}_{v_{\ell}}$ which shows that the matrix assignment for $\mathbold{ c}_{v_{\ell}}$ is nonsingular.

On the other hand, if the leading coefficeint of $A(\lambda)$ and $D(\lambda)$ are nonsingular, then the matrix assignment for $\mathbold{ c}_{v_{h}}$ and $\mathbold{ c}_{v_{\ell}}$ are nonsingular irrespective of $m$ and $k$ being odd or even.

Hence $\mathbb{L}(\lambda)$ is a Rosenbrock linearization of $\mathcal{S}(\lambda)$.
\end{proof}

\begin{example}
     Let $G(\lambda)=D(\lambda)+B^{T}A(\lambda)^{-1}B\in\mathbb{C}(\lambda)^{r\times r}$ where $A(\lambda)=\lambda^{3}A_{3}+\lambda^{2}A_{2}+\lambda  A_{1}+A_{0}\in\mathbb{C}[\lambda]^{n\times n}$, $D(\lambda)=\lambda^{5}D_{5}+\lambda^{4}D_{4}+\lambda^{3}D_{3}+\lambda^{2}D_{2}+\lambda  D_{1}+D_{0}\in\mathbb{C}[\lambda]^{r\times r}$, $C\in\mathbb{C}^{r\times n}$, and $B\in\mathbb{C}^{n\times r}$. Let us consider $h=2$, $\mathbold{w}_{h}=(1,2,0)$, $\mathbold{v}_{h}=(-3)$, $\mathbold{c}_{w_{h}}=(1)$, $\mathbold{c}_{v_{h}}=\phi$, $\mathbold{t}_{w_{h}}=(0)$, and, $\mathbold{t}_{v_{h}}=\phi$. Similarly let us consider $\ell=2$, $\mathbold{w}_{\ell}=(1,2,0)$, $\mathbold{v}_{\ell}=(-3)$, $\mathbold{c}_{w_{\ell}}=(1)$, $\mathbold{c}_{v_{\ell}}=\phi$, $\mathbold{t}_{w_{\ell}}=(0)$, and, $\mathbold{t}_{v_{\ell}}=(-5)$. Let $X_{A}$, $Y_{A}$, $X_{D}$ and $Y_{D}$ are arbitrary nonsingular symmetric matrices. Then, GFPR $\mathbb{L}(\lambda)$ of $G(\lambda)$ given by,
     {\small
     \begin{align*}
         \mathbb{L}(\lambda)
           =\left[\begin{array}{ccc|ccccc}
              \lambda A_{3}+ A_{2} & A_{1} & -X_{A} & 0 & 0 & 0 & 0 & 0 \\
              A_{1} & \lambda A_{1}+A_{0} & \lambda X_{A} & 0 & 0 & 0 & B & 0 \\
              -X_{A} & \lambda X_{A} & 0 & 0 & 0 & 0 & 0 & 0 \\\hline
              0 & 0 & 0 & 0 & -Y_{D} & \lambda Y_{D} & 0 & 0 \\
              0 & 0 & 0 & -Y_{D} & \lambda D_{5}-D_{4} & \lambda D_{4} & 0 & 0 \\
              0 & 0 & 0 & \lambda Y_{D} & \lambda D_{4} & \lambda D_{3}+D_{2} & D_{1} & -X_{D}\\
              0 & B^{T} & 0 & 0 & 0 & D_{1} & -\lambda D_{1} + D_{0} & \lambda X_{D} \\
              0 & 0 & 0 & 0 & 0 & -X_{D} & \lambda X_{D} & 0
          \end{array}\right]
     \end{align*}}
     is a symmetric Rosenbrock linearization of $G(\lambda)$.
\end{example}

\subsection{T-even linearizations}
A rational matrix function $G(\lambda)$ is $T-$even if $G(-\lambda)^{T}=G(\lambda).$ A rational matrix function has a $T-$even realization if $G(\lambda)=D(\lambda)+CA(\lambda)^{-1}B$ with $A(\lambda)$ and $D(\lambda)$ $T-$even and $C=B^{T}$. A system matrix $\mathcal{S}(\lambda)=\left[\begin{array}{c|c} A(\lambda) & B \\\hline C & D(\lambda)\end{array}\right]$ is $T-$even if $A(\lambda)$ and $D(\lambda)$ $T-$even and $C=B^{T}$. It is easy to check that $G(\lambda)=D(\lambda)+CA(\lambda)^{-1}B$ is a $T-$even realization of $G(\lambda)$ whenever $\mathcal{S}(\lambda)=\left[\begin{array}{c|c} A(\lambda) & B \\\hline C & D(\lambda)\end{array}\right]$ is a $T-$even. In this section, we have constructed $T-$even linarization of $G(\lambda)$ having a $T-$even realization.

\begin{definition}[\cite{BFii2014}]
    A matrix $Q_{A}\in\mathbb{C}^{mn\times mn}$is said to be quasi-identity matrix if $Q_{A}=\epsilon_{1}I_{n}\oplus\ldots\oplus\epsilon_{m}I_{n}$, where $\epsilon_{i}\in\{\pm1\}$ for $i=1,2,\ldots,m$. We referto $\epsilon_{j}$ as the $j-$th parameter of $Q_{A}$.
\end{definition}

Next we recall the folowing theorem which is a particluar case of \cite[Theorem 4.15]{BFii2014}.

\begin{theorem}\cite{BFii2014}\label{evenoddpolynomialpencil}
    Let $h\in\{0,1,\ldots, m-1\}$ be even . Consider a simple admissible tuple $\mathbold{ w}$ of $\{0,1,\ldots,h\}$ with $\mathbold{ c}_{w}$ as the symmetric complement. Similarly for any admissible tuple $\mathbold{ z}+m$ of $\{0,1,\ldots,m-h-1\}$ with $\mathbold{ c}_{z}+m$ as the symmetric complement. Let $L(\lambda)=\left(\lambda M^{A}_{z}-M^{A}_{w}\right)M^{A}_{c_{w}}M^{A}_{c_{z}}.$ Then, up to multiplication by $-1$, there exists a unique quasi-identity matrix $Q$ such that $QL(\lambda)$ is $T-$even (resp., $T-$odd) when $A(\lambda)$ is $T-$even (resp., $T-$odd).  
\end{theorem}

More information on the construction of the quasi identity matrix $Q$ can be found in \cite[Algorithm 4.14]{BFii2014}. In the next theorem, we construct a $T-$even linearization of $G(\lambda)$.

\begin{theorem}\label{tevenproof}
    Let $G(\lambda)$ be a rational matrix with $T-$ even realization and $\mathcal{S}(\lambda)$ be the corresponding system matrix. Let $h\in\{0,1,\ldots, m-1\}$ is even, $\mathbold{ w}$ be the simple admissible tuple of $\{0,1,\ldots,h\}$ and, $\mathbold{ c}_{w}$ be the symmetric complement of $\mathbold{ w}$. Let $\mathbold{ z}_{h}+m$ be the admissible tuple of $\{0,1,\ldots,m-h-1\}$ and $\mathbold{ c}_{z_{h}}+m$ be the symmetric complement of $\mathbold{ z}_{h}+m$. Similarly, let $\ell\in\{0,1,\ldots, k-1\}$ is even, $\mathbold{ v}$ is the simple admissible tuple of $\{0,1,\ldots,\ell\}$, and, $\mathbold{ c}_{v}$ be the symmetric complement of $\mathbold{ v}$. Let $\mathbold{ z}_{\ell}+k$ be the admissible tuple of $\{0,1,\ldots,k-\ell-1\}$ and $\mathbold{ c}_{z_{\ell}}+k$ be the symmetric complement of $\mathbold{ z}_{\ell}+k$. Then there exist unique quasi identity matrices $Q_{A}$ and $Q_{D}$ such that,
    \[
    \mathbb{L}(\lambda)=\left[\begin{array}{c|c}
                               Q_{A}L_{A}(\lambda) & e_{m-i_{0}(\mathbold{ w})}e^{T}_{k-c_{0}(\mathbold{ v},\mathbold{ c}_{v})}\otimes B \\\hline
                               e_{k-i_{0}(\mathbold{ v})}e^{T}_{m-c_{0}(\mathbold{ w},\mathbold{ c}_{w})}\otimes B^{T} & Q_{D}L_{D}(\lambda)
                        \end{array}\right]
    \]
    is $T-$even, where $L_{A}(\lambda)=\left(\lambda M^{A}_{\mathbold{ z}_{h}}-M^{A}_{\mathbold{ w}}\right)M^{A}_{\mathbold{ c}_{w}}M^{A}_{\mathbold{ c}_{z_{h}}}$ and $L_{D}(\lambda)=\left(\lambda M^{D}_{\mathbold{ z}_{\ell}}-M^{D}_{\mathbold{ v}}\right)M^{D}_{\mathbold{ c}_{v}}M^{D}_{\mathbold{ c}_{z_{\ell}}}$ are the pencils as defined in Theorem \ref{evenoddpolynomialpencil} for the matrix polynomials $A(\lambda)$ and $D(\lambda)$ respectively. 
       When the leading coefficient of $A(\lambda)$ is singular, assume $\mathrm{Ind}(\mathbold{ z}_{h}+m)=0$ and when the leading coeffiecint of $D(\lambda)$ is singular, assume $\mathrm{Ind}(\mathbold{ z}_{\ell}+k)=0$. Then $\mathbb{L}(\lambda)$ is a Rosenbrock linearization of $G(\lambda)$.
\end{theorem}

\begin{proof}
    It is given that $h$ is even and $\mathbold{ w}$ is the simple admissible tuple of $\{0,1,\ldots,h\}$. So, we have, $\mathbold{ w}=(h-1:h,\ldots,3:4,1:2,0)$ and $\mathbold{ c}_{w}=(h-1,h-3,\ldots,1)$. Therefore, when $h=0$, $i_{0}(\mathbold{ w})=c_{0}(\mathbold{ w},\mathbold{ c_{w}})=0$ and when $h>0$, $i_{0}(\mathbold{ w})=c_{0}(\mathbold{ w},\mathbold{ c_{w}})=1$. With similar arguments, since $\ell$ is even and $\mathbold{ v}$ is the simple admissible tuple of $\{0,1,\ldots,\ell\}$, it can be shown that $i_{0}(\mathbold{ v})=c_{0}(\mathbold{ v},\mathbold{ c_{v}})$ for all $\ell\geq0$. From Theorem \ref{evenoddpolynomialpencil}, $Q_{A}L_{A}(\lambda)$ and $Q_{D}L_{D}(\lambda)$ are $T-$even. Consequently $\mathbb{L}(\lambda)$ is $T-$even.

   As $h$ and $\ell$ are even, we have $\mathbold{ c}_{w}=(h-1,h-3,\ldots,1)$ and $\mathbold{ c}_{v}=(\ell -1,\ell -3,\ldots,1)$. Thus $0\notin\mathbold{ c}_{w}$ and $0\notin\mathbold{ c}_{v}$ and hence the matrix assignments for both $\mathbold{ c}_{w}$ and $\mathbold{ c}_{v}$ are nonsingular. Therefore, taking $\sigma=\mathbold{ w}$, $\tau=\mathbold{ z}_{h}$, $\sigma_{1}=\tau_{1}=\phi$, $\sigma_{2}=\mathbold{ c}_{w}$ and $\tau_{2}=\mathbold{ c}_{\mathbold{ z}_{h}}$ and $\gamma=\mathbold{ v}$, $\delta=\mathbold{ z}_{\ell}$, $\gamma_{1}=\delta_{1}=\phi$, $\gamma_{2}=\mathbold{ c}_{v}$ and $\delta_{2}=\mathbold{ c}_{\mathbold{ z}_{\ell}}$, from Theorem \ref{gfprproof} that $\mathbb{L}(\lambda)$ is a Rosenbrock linearization of $G(\lambda)$ if the matrix assignment for $\mathbold{ c}_{\mathbold{ z}_{h}}$ and $\mathbold{ c}_{\mathbold{ z}_{h}}$ are nonsingular. If the leading coefficient of $A(\lambda)$ and $D(\lambda)$ are nonsingular then the matrix assignment for $\mathbold{ c}_{\mathbold{ z}_{h}}$ and $\mathbold{ c}_{\mathbold{ z}_{h}}$ are nonsingular. If the leading coeffient of $A(\lambda)$ is singular and $\mathrm{Ind}(\mathbold{ z_{h}}+m)=0$, then $0\notin\mathbold{ c}_{\mathbold{ z}}+m\implies -m\notin\mathbold{ c}_{z}$ and hence the matrix assignment for $\mathbold{ c}_{z}$ is nonsingular. Similarly, if the leading coefficient of $D(\lambda)$ is singular and $\mathrm{Ind}(\mathbold{ z_{\ell}}+k)=0$, then $0\notin\mathbold{ c}_{\mathbold{ z_{\ell}}}+k\implies -k\notin\mathbold{ c}_{z_{\ell}}$ and hence the matrix assignment for $\mathbold{ c}_{z_{\ell}}$ is nonsingular. Thus $\mathbb{L}(\lambda)$ is a $T-$even Rosenbrock linearization of $G(\lambda)$.
\end{proof}

\begin{example}
Let $G(\lambda)=D(\lambda)+B^{T}A(\lambda)^{-1}B\in\mathbb{C}(\lambda)^{r\times r}$ where $A(\lambda)=\lambda^{5}A_{5}+\lambda^{4}A_{4}+\lambda^{3}A_{3}+\lambda^{2}A_{2}+\lambda A_{1}+A_{0}\in\mathbb{C}[\lambda]^{n\times n}$, $D(\lambda)=\lambda^{4}D_{4}+\lambda^{3}D_{3}+\lambda^{2}D_{2}+\lambda  D_{1}+D_{0}\in\mathbb{C}[\lambda]^{r\times r}$, $C\in\mathbb{C}^{r\times n}$, and $B\in\mathbb{C}^{n\times r}$. If we consider $h=2$, $\mathbold{w}=(1,2,0)$, $\mathbold{z}_{h}=(-4,-3,-5)$, $\mathbold{c}_{w}=(1)$, and, $\mathbold{c}_{z_{h}}=(-4)$. Similarly $\ell=0$, $\mathbold{v}=(0)$, $\mathbold{z}_{\ell}=(-4,-3,-2,-1)$, $\mathbold{c}_{v}=\phi$, and, $\mathbold{c}_{z_{\ell}}=(-4,-3,-2,-4,-3,-4)$. Also take $Q_{A}=\mathrm{diag}(I_{n},I_{n},-I_{n},I_{n},-I_{n})$ and $Q_{D}=\mathrm{diag}(I_{r},-I_{r},I_{r},-I_{r})$. Then,
     {\scriptsize \[
      \mathbb{L}(\lambda)
        =\left[\begin{array}{ccccc|cccc}
           0 & -I_{n} & \lambda I_{n} & 0 & 0 & 0 & 0 & 0 & 0 \\
           -I_{n} & \lambda A_{5}-A_{4} & \lambda A_{4} & 0 & 0 & 0 & 0 & 0 & 0  \\
           -\lambda I_{n} & -\lambda A_{4} & -\lambda A_{3} - A_{3} & -A_{1} & I_{n} & 0 & 0 & 0 & 0 \\
           0 & 0 & A_{1} & -\lambda A_{1}+A_{0} & \lambda I_{n} & 0 & 0 & 0 & B \\
           0 & 0 & I_{n} & -\lambda I_{n} & 0 & 0 & 0 & 0 & 0 \\\hline
           0 & 0 & 0 & 0 & 0 & 0 & 0 & -D_{4} & \lambda D_{4} \\
           0 & 0 & 0 & 0 & 0 & 0 & D_{4} & -\lambda D_{4} + D_{3} & -\lambda D_{3} \\
           0 & 0 & 0 & 0 & 0 & -D_{4} & \lambda D_{4}-D_{3} & \lambda D_{3} -D_{2} & \lambda D_{2} \\
           0 & 0 & 0 & B^{T} & 0 & -\lambda D_{4} & -\lambda D_{3} & -\lambda D_{2} & -\lambda D_{1} - D_{0}
        \end{array}\right]
      \]}
      is a $T-$even Rosenbrock linearization of $G(\lambda)$ whenever $D_{4}$ is nonsingular.
\end{example}

\subsection{T-odd linearizations}
A rational matrix function $G(\lambda)$ is said to be $T-$odd if $G(-\lambda)^{T}=-G(\lambda).$ A rational matrix function is said to have a $T-$odd realization if $G(\lambda)=D(\lambda)+C(A(\lambda)^{-1}B$ with $A(\lambda)$ and $D(\lambda)$ $T-$odd and $C=-B^{T}$. A system matrix $\mathcal{S}(\lambda)=\left[\begin{array}{c|c} A(\lambda) & B \\\hline C & D(\lambda)\end{array}\right]$ is said to be $T-$odd if $A(\lambda)$ and $D(\lambda)$ are $T-$odd and $C=-B^{T}$. It is easy to check that $G(\lambda)=D(\lambda)+C(A(\lambda)^{-1}B$ is a $T-$odd realization of $G(\lambda)$ whenever $\mathcal{S}(\lambda)=\left[\begin{array}{c|c} A(\lambda) & B \\\hline C & D(\lambda)\end{array}\right]$ is $T-$odd.
 In this section we construct a $T-$even linearization of $G(\lambda)$ whenever it has a $T-$even linearization.

\begin{theorem}\label{toddproof}
    Let $G(\lambda)$ be a rational matrix with $T-$odd realization and $\mathcal{S}(\lambda)$ be the corresponding system matrix. Let $h\in\{0,1,\ldots,m-1\}$ is even, $\mathbold{ w}$ be the simple admissible tuple of $\{0,1,\ldots,h\}$ and, $\mathbold{ c}_{w}$ be the symmetric complement of $\mathbold{ w}$. Let $\mathbold{ z}_{h}+m$ be the admissible tuple of $\{0,1,\ldots,m-h-1\}$ and $\mathbold{ c}_{z_{h}}+m$ be the symmetric complement of $\mathbold{ z}_{h}+m$. Similarly, let $\ell\in\{0,1,\ldots, k-1\}$ is even, $\mathbold{ v}$ is the simple admissible tuple of $\{0,1,\ldots,\ell\}$, and, $\mathbold{ c}_{v}$ be the symmetric complement of $\mathbold{ v}$. Let $\mathbold{ z}_{\ell}+k$ be the admissible tuple of $\{0,1,\ldots,k-\ell-1\}$ and $\mathbold{ c}_{z_{\ell}}+k$ be the symmetric complement of $\mathbold{ z}_{\ell}+k$. Then there exist unique quasi identity matrices $Q_{A}$ and $Q_{D}$ such that,
    \[
    \mathbb{L}(\lambda)=\left[\begin{array}{c|c}
                               Q_{A}L_{A}(\lambda) & e_{m-i_{0}(\mathbold{ w})}e^{T}_{k-c_{0}(\mathbold{ v},\mathbold{ c}_{v})}\otimes B \\\hline
                               -e_{k-i_{0}(\mathbold{ v})}e^{T}_{m-c_{0}(\mathbold{ w},\mathbold{ c}_{w})}\otimes B^{T} & Q_{D}L_{D}(\lambda)
                        \end{array}\right]
    \]
    is $T-$odd, where $L_{A}(\lambda)=\left(\lambda M^{A}_{\mathbold{ z}_{h}}-M^{A}_{\mathbold{ w}}\right)M^{A}_{\mathbold{ c}_{w}}M^{A}_{\mathbold{ c}_{z}}$ and $L_{D}(\lambda)=\left(\lambda M^{D}_{\mathbold{ z}_{\ell}}-M^{D}_{\mathbold{ v}}\right)M^{D}_{\mathbold{ c}_{v}}M^{D}_{\mathbold{ c}_{z_{\ell}}}$ are the pencils as defined in Theorem \ref{evenoddpolynomialpencil} for the matrix polynomials $A(\lambda)$ and $D(\lambda)$ respectively. 
       When the leading coefficient of $A(\lambda)$ is singular, assume $\mathrm{Ind}(\mathbold{ z}_{h}+m)=0$ and when the leading coeffiecint of $D(\lambda)$ is singular, assume $\mathrm{Ind}(\mathbold{ z}_{\ell}+k)=0$. Then $\mathbb{L}(\lambda)$ is a Rosenbrock linearization of $G(\lambda)$.
\end{theorem}

\begin{proof}
    It is given that $h$ is even and $\mathbold{ w}$ is the simple admissible tuple of $\{0,1,\ldots,h\}$. So, we have, $\mathbold{ w}=(h-1:h,\ldots,3:4,1:2,0)$ and $\mathbold{ c}_{w}=(h-1,h-3,\ldots,1)$. Therefore, when $h=0$, $i_{0}(\mathbold{ w})=c_{0}(\mathbold{ w},\mathbold{ c_{w}})=0$ and when $h>0$, $i_{0}(\mathbold{ w})=c_{0}(\mathbold{ w},\mathbold{ c_{w}})=1$. With similar arguments, since $\ell$ is even and $\mathbold{ v}$ is the simple admissible tuple of $\{0,1,\ldots,\ell\}$, it can be shown that $i_{0}(\mathbold{ v})=c_{0}(\mathbold{ v},\mathbold{ c_{v}})$ for all $\ell\geq0$. From Theorem \ref{evenoddpolynomialpencil}, $Q_{A}L_{A}(\lambda)$ and $Q_{D}L_{D}(\lambda)$ are $T-$odd. Consequently $\mathbb{L}(\lambda)$ is $T-$odd.

   As $h$ and $\ell$ are even, we have $\mathbold{ c}_{w}=(h-1,h-3,\ldots,1)$ and $\mathbold{ c}_{v}=(\ell -1,\ell -3,\ldots,1)$. Thus $0\notin\mathbold{ c}_{w}$ and $0\notin\mathbold{ c}_{v}$ and hence the matrix assignments for both $\mathbold{ c}_{w}$ and $\mathbold{ c}_{v}$ are nonsingular. Therefore, taking $\sigma=\mathbold{ w}$, $\tau=\mathbold{ z}_{h}$, $\sigma_{1}=\tau_{1}=\phi$, $\sigma_{2}=\mathbold{ c}_{w}$ and $\tau_{2}=\mathbold{ c}_{\mathbold{ z}_{h}}$ and $\gamma=\mathbold{ v}$, $\delta=\mathbold{ z}_{\ell}$, $\gamma_{1}=\delta_{1}=\phi$, $\gamma_{2}=\mathbold{ c}_{v}$ and $\delta_{2}=\mathbold{ c}_{\mathbold{ z}_{\ell}}$, from Theorem \ref{gfprproof} that $\mathbb{L}(\lambda)$ is a Rosenbrock linearization of $G(\lambda)$ if the matrix assignment for $\mathbold{ c}_{\mathbold{ z}_{h}}$ and $\mathbold{ c}_{\mathbold{ z}_{h}}$ are nonsingular. If the leading coefficient of $A(\lambda)$ and $D(\lambda)$ are nonsingular then the matrix assignment for $\mathbold{ c}_{\mathbold{ z}_{h}}$ and $\mathbold{ c}_{\mathbold{ z}_{h}}$ are nonsingular. If the leading coeffient of $A(\lambda)$ is singular and $\mathrm{Ind}(\mathbold{ z_{h}+m}=0$, then $0\notin\mathbold{ c}_{\mathbold{ z}}+m\implies -m\notin\mathbold{ c}_{z}$ and hence the matrix assignment for $\mathbold{ c}_{z}$ is nonsingular. Similarly, if the leading coefficient of $D(\lambda)$ is singular and $\mathrm{Ind}(\mathbold{ z_{\ell}+k}=0$, then $0\notin\mathbold{ c}_{\mathbold{ z_{\ell}}}+k\implies -k\notin\mathbold{ c}_{z_{\ell}}$ and hence the matrix assignment for $\mathbold{ c}_{z_{\ell}}$ is nonsingular. Thus $\mathbb{L}(\lambda)$ is a $T-$odd Rosenbrock linearization of $G(\lambda)$.
\end{proof}

\begin{example}
Let $G(\lambda)=D(\lambda)+B^{T}A(\lambda)^{-1}B\in\mathbb{C}(\lambda)^{r\times r}$ where $A(\lambda)=\lambda^{5}A_{5}+\lambda^{4}A_{4}+\lambda^{3}A_{3}+\lambda^{2}A_{2}+\lambda A_{1}+A_{0}\in\mathbb{C}[\lambda]^{n\times n}$, $D(\lambda)=\lambda^{4}D_{4}+\lambda^{3}D_{3}+\lambda^{2}D_{2}+\lambda  D_{1}+D_{0}\in\mathbb{C}[\lambda]^{r\times r}$, $C\in\mathbb{C}^{r\times n}$, and $B\in\mathbb{C}^{n\times r}$. If we consider $h=2$, $\mathbold{w}=(1,2,0)$, $\mathbold{z}_{h}=(-4,-3,-5)$, $\mathbold{c}_{w}=(1)$, and, $\mathbold{c}_{z_{h}}=(-4)$. Similarly $\ell=0$, $\mathbold{v}=(0)$, $\mathbold{z}_{\ell}=(-4,-3,-2,-1)$, $\mathbold{c}_{v}=\phi$, and, $\mathbold{c}_{z_{\ell}}=(-4,-3,-2,-4,-3,-4)$. Also take $Q_{A}=\mathrm{diag}(I_{n},-I_{n},I_{n},-I_{n},-I_{n})$ and $Q_{D}=\mathrm{diag}(I_{r},-I_{r},I_{r},-I_{r})$. Then,
     {\scriptsize \[
      \mathbb{L}(\lambda)
        =\left[\begin{array}{ccccc|cccc}
           0 & -I_{n} & \lambda I_{n} & 0 & 0 & 0 & 0 & 0 & 0 \\
           I_{n} & -\lambda A_{5}+A_{4} & -\lambda A_{4} & 0 & 0 & 0 & 0 & 0 & 0  \\
           \lambda I_{n} & \lambda A_{4} & \lambda A_{3} + A_{2} & A_{1} & -I_{n} & 0 & 0 & 0 & 0 \\
           0 & 0 & -A_{1} & \lambda A_{1}-A_{0} & -\lambda I_{n} & 0 & 0 & 0 & 1 \\
           0 & 0 & I_{n} & -\lambda I_{n} & 0 & 0 & 0 & 0 & 0 \\\hline
           0 & 0 & 0 & 0 & 0 & 0 & 0 & -D_{4} & \lambda D_{4} \\
           0 & 0 & 0 & 0 & 0 & 0 & D_{4} & -\lambda D_{4} + D_{3} & -\lambda D_{3} \\
           0 & 0 & 0 & 0 & 0 & -D_{4} & \lambda D_{4}-D_{3} & \lambda D_{3} -D_{2} & \lambda D_{2} \\
           0 & 0 & 0 & -B^{T} & 0 & -\lambda D_{4} & -\lambda D_{3} & -\lambda D_{2} & -\lambda D_{1} - D_{0}
        \end{array}\right]
      \]}
      is a $T-$even Rosenbrock linearization of $G(\lambda)$ whenever $D_{4}$ is nonsingular.
\end{example}

\subsection{Skew-symmetric linearizations}

A rational matrix $G(\lambda)$ is said to be skew-symmetric if $G(\lambda)^{T}=-G(\lambda)$. A rational matrix $G(\lambda)$ is said to have a skew symmetric realization if $G(\lambda)=D(\lambda)+CA(\lambda)^{-1}B$ if both $A(\lambda)$ and $D(\lambda)$ are skew-symmetric and $C^{T}=B$. Observe that a realization of a rational matrix $G(\lambda)$ is skew-symmetric if and only if the corresponding system matrix $\mathcal{S}(\lambda)$ is skew-symmetric. Here we cosntruct a skew-symmetric Rosenbrock linearization of $G(\lambda)$.

Consider a permutation $\alpha$ $\{0,1,\ldots,k\}$ where $k\geq 0$. If for $s\in\{0,1,\ldots,k-1\}$, $csf{\alpha}$ contains a string $(s:t)$ with $s<t$, then $s\in\{0,1,\ldots,k-1\}$ $s$ is called a right index of type$-1$ relative to $\alpha$ \cite{BFii2014}.

\begin{definition}[\cite{BFii2014}, Associated simple tuple]
Consider a permutation $\alpha$ of $\{0,1,\ldots,k\}$ for some $k\geq0$ with $csf(\alpha)=(\mathbold{ b}_{d},\mathbold{ b}_{d-1},\ldots,\mathbold{ b}_{1})$, where $\mathbold{ b}_{i}=(a_{i-1}+1:a_{i})$ for $i=2:d$ and $\mathbold{ b}_{1}=(0:a_{1}).$ For a right index of type-$1$ $s$ relative to $\alpha$ the simple tuple associated with $(\alpha,s)$ is denoted by $z_{r}(\alpha,s)$ and is given by
\begin{itemize}
    \item $z_{r}(\alpha,s)=\left(\mathbold{ b}_{d},\mathbold{ b}_{d-1},\ldots,\mathbold{ b}_{h+1},\tilde{\mathbold{ b}}_{h},\tilde{\mathbold{ b}}_{h-1},\mathbold{ b}_{h-2},\ldots,\mathbold{ b}_{1}\right)$ if $s=a_{h-1}+1\neq0$, where $\tilde{\mathbold{ b}}_{h}=(a_{h-1}+2:a_{h})$ and $\tilde{\mathbold{ b}}_{h-1}=(a_{h-2}+1:a_{h-1}+1)$.

    \item $z_{r}(\alpha,s)=\left(\mathbold{ b}_{d},\mathbold{ b}_{d-1},\ldots,\mathbold{ b}_{2},\tilde{\mathbold{ b}}_{1},\tilde{\mathbold{ b}}_{0}\right)$ if $s=0$, where $\tilde{\mathbold{ b}}_{1}=(1:a_{1})$ and $\tilde{\mathbold{ b}}_{0}=(0)$.
\end{itemize}
\end{definition}

\begin{definition}[\cite{BFii2014}, Type$-1$ index tuple]
    Consider a permutation $\alpha$ of $\{0,1,\ldots,k\}$, $k\geq0$, and let $\beta=(s_{1},\ldots,s_{r})$ be an index tuple of type$-1$ relative to $\alpha$ if, for $i=1,2,\ldots,r$, $s_{i}$ is a right index of type$-1$ relative to $z_{r}(\alpha,(s_{1},\ldots,s_{i-1}))$, where $z_{r}(\alpha,(s_{1},\ldots,s_{i-1}))=z_{r}(z_{r}(\alpha,(s_{1},\ldots,s_{i-2})),s_{i-1})$ for $i>2$.
\end{definition}

Next we recall the following theorem which is a particular case of \cite[Theorem 3.15]{BFii2014}.

\begin{theorem}[\cite{BFii2014}]\label{skewsymmetricpolypencil}
Let $A(\lambda)$ be skew symmetric and let $h\in\{0,1,\ldots, m-1\}$ be even. Let $\mathbold{ w}$ be the simple admissible tuple of $\{0,1,\ldots,h\}$ and $\mathbold{ c}_{w}$ be the symmetric complement of $\mathbold{ w}$. Let $\mathbold{ z}+m$ be any admissible tuple of $\{0,1,\ldots,m-h-1\}$ and $\mathbold{ c}_{z}+m$ be the symmetric complement of $\mathbold{ z}+m.$ Let $\mathbold{ t}_{w}$ containing indices from $\{0,1,\ldots,h-1\}$ and $\mathbold{ t}_{z}+m$ containing indices from $\{0,1,\ldots,m-h-2\}$ be the right index tuples of type$-1$ relative to $\mathrm{rev}(\mathbold{ w})$ and $\mathrm{rev}(\mathbold{ z}+m)$, respectively. Consider
\[
L(\lambda)=M^{A}_{\mathrm{rev}(\mathbold{ t}_{z}}M^{A}_{\mathrm{rev}(\mathbold{ t}_{w}}\left(\lambda M^{A}_{\mathbold{ z}}-M^{A}_{\mathbold{ w}}\right)M^{A}_{\mathbold{ c}_{w}}M^{A}_{\mathbold{ t}_{w}}M^{A}_{\mathbold{ c}_{z}}M^{A}_{\mathbold{ t}_{z}}.
\]
Then, upto multiplication by $-1$, there exist a unique quasi-identity matrix $Q$ such that $QL(\lambda)$ is skew-symmetric.
\end{theorem}

We now construct skew-symmetric Rosenbrock linearizations of $G(\lambda)$.

\begin{theorem}
    Let $G(\lambda)$ be a rational matrix with skew-symmetric realization and $\mathcal{S}(\lambda)$ be the corresponding system matrix. Let $h\in\{0,1,\ldots, m-1\}$ be even. Let $\mathbold{ w}$ be the simple admissible tuple of $\{0,1,\ldots,h\}$ and $\mathbold{ c}_{w}$ be the symmetric complement of $\mathbold{ w}$. Let $\mathbold{ z}_{h}+m$ be any admissible tuple of $\{0,1,\ldots,m-h-1\}$ and $\mathbold{ c}_{{z}_{h}}+m$ be the symmetric complement of $\mathbold{ z}_{h}+m.$ Let $\mathbold{ t}_{w}$ containing indices from $\{0,1,\ldots,h-1\}$ and $\mathbold{ t}_{{z}_{h}}+m$ containing indices from $\{0,1,\ldots,m-h-2\}$ be the right index tuples of type$-1$ relative to $\mathrm{rev}(\mathbold{ w})$ and $\mathrm{rev}(\mathbold{ z}_{h}+m)$, respectively. Similarly, let $\ell\in\{0,1,\ldots, k-1\}$ be even. Let $\mathbold{ v}$ be the simple admissible tuple of $\{0,1,\ldots,\ell\}$ and $\mathbold{ c}_{v}$ be the symmetric complement of $\mathbold{ v}$. Let $\mathbold{ z}_{\ell}+k$ be any admissible tuple of $\{0,1,\ldots,k-\ell-1\}$ and $\mathbold{ c}_{{z}_{\ell}}+k$ be the symmetric complement of $\mathbold{ z}_{\ell}+k.$ Let $\mathbold{ t}_{v}$ containing indices from $\{0,1,\ldots,\ell-1\}$ and $\mathbold{ t}_{{z}_{\ell}}+k$ containing indices from $\{0,1,\ldots,k-\ell-2\}$ be the right index tuples of type$-1$ relative to $\mathrm{rev}(\mathbold{ v})$ and $\mathrm{rev}(\mathbold{ z}_{\ell}+k)$, respectively. Then there exist unique quasi-identity matrices $Q_{A}$ and $Q_{D}$ such that,
    \[
    \mathbb{L}(\lambda)=\left[\begin{array}{c|c}
                               Q_{A}L_{A}(\lambda) & e_{m-i_{0}(rev(\mathbold{ t}_{w}),\mathbold{ w})}e^{T}_{k-c_{0}(\mathbold{ v},\mathbold{ c}_{v},\mathbold{ t}_{v}}\otimes B \\\hline
                               e_{k-i_{0}(rev(\mathbold{ t}_{v},\mathbold{ v}))}e^{T}_{m-c_{0}(\mathbold{ w},\mathbold{ c}_{w},\mathbold{ t}_{w})}\otimes B^{T} & Q_{D}L_{D}(\lambda)
                        \end{array}\right]
    \]
    is skew-symmetric, where $$L_{A}(\lambda)=M^{A}_{rev(\mathbold{ t}_{{z}_{h}})}M^{A}_{rev(\mathbold{ t}_{w})}\left(\lambda M^{A}_{\mathbold{ z}}-M^{A}_{\mathbold{ w}}\right)M^{A}_{\mathbold{ c}_{w}}M^{A}_{\mathbold{ t}_{w}}M^{A}_{{\mathbold{ c}}_{{z}_{\ell}}}M^{A}_{\mathbold{ t}_{{z}_{h}}}$$ 
    
    and $$L_{D}(\lambda)=M^{D}_{rev(\mathbold{ t}_{{z}_{\ell}})}M^{D}_{rev(\mathbold{ t}_{v})}\left(\lambda M^{D}_{\mathbold{ z}_{h}}-M^{D}_{\mathbold{ w}}\right)M^{D}_{\mathbold{ c}_{w}}M^{D}_{\mathbold{ t}_{w}}M^{D}_{\mathbold{ c}_{{z}_{\ell}}}M^{D}_{\mathbold{ t}_{{z}_{\ell}}}$$ are the pencils as defined in Theorem \ref{skewsymmetricpolypencil} for the matrix polynomials $A(\lambda)$ and $D(\lambda)$ respectively.

    When the leading coefficient of $A(\lambda)$ is singular, assume $\mathrm{Ind}(\mathbold{ z}_{h}+m)=0$. Further, assume that $0\notin\mathbold{ t}_{w}$ (resp., $-m\notin\mathbold{ t}_{z_{h}})$ when $A_{0}$ (resp., $A_{m})$ is singular. Similarly when the leading coefficient of $D(\lambda)$ is singular, assume $\mathrm{Ind}(\mathbold{ z}_{\ell}+k)=0$. Further, assume that $0\notin\mathbold{ t}_{v}$ (resp., $-k\notin\mathbold{ t}_{z_{\ell}})$ when $D_{0}$ (resp., $D_{k})$ is singular. Then $\mathbb{L}(\lambda)$ is a Rosenbrock linearization of $G(\lambda)$.
\end{theorem}

\begin{proof}
    It is given that $h$ is even and $\mathbold{ w}$ is the simple admissible tuple of $\{0,1,\ldots,h\}$. When $h=0$, we have, $\mathbold{ w}=0$ and $\mathbold{ c}_{w}=\mathbold{ t}_{w}=\phi$ so that $i_{0}(rev(\mathbold{ t}_{w}),\mathbold{ w})=c_{0}(\mathbold{ w},\mathbold{ c_{w}},\mathbold{ t}_{w})=0$. When $h>0$, $\mathbold{ w}=(h-1:h,\ldots,3:4,1:2,0)$ and $\mathbold{ c}_{w}=(h-1,h-3,\ldots,1)$ and hence $c_{0}(\mathbold{ w},\mathbold{ c_{w}},\mathbold{ t}_{w})=2+c_{2}(\mathbold{ t}_{w})$ and $i_{0}(rev(\mathbold{ t}_{w}),\mathbold{ w})=2+i_{2}(rev(\mathbold{ t}_{w})$ so that $i_{0}(rev(\mathbold{ t}_{w}),\mathbold{ w})=c_{0}(\mathbold{ w},\mathbold{ c_{w}},\mathbold{ t}_{w})$.  With similar arguments we can show that $i_{0}(rev(\mathbold{ t}_{v}),\mathbold{ v})=c_{0}(\mathbold{ v},\mathbold{ c_{v}},\mathbold{ t}_{v})$. From Theorem \ref{skewsymmetricpolypencil}, $Q_{A}L_{A}(\lambda)$ and $Q_{D}L_{D}(\lambda)$ are $T-$even. Consequently $\mathbb{L}(\lambda)$ is skew-symetric.

    When $A_{0}$ (resp., $A_{m}$) is singular, $0\notin\mathbold{ t}_{w}$ (resp., $-m\notin\mathbold{ t}_{z_{h}}$) so the matrix assignments for $\mathbold{ t}_{w}$, $rev(\mathbold{ t}_{w})$, $\mathbold{ t}_{{z}_{h}}$ and $rev(\mathbold{ t}_{z_{h}})$ are nonsingular. Similarly, when $D_{0}$ (resp., $D_{k}$) is singular, $0\notin\mathbold{ t}_{v}$ (resp., $-k\notin\mathbold{ t}_{z_{\ell}}$) so the matrix assignments for $\mathbold{ t}_{v}$, $rev(\mathbold{ t}_{v})$, $\mathbold{ t}_{{z}_{\ell}}$ and $rev(\mathbold{ t}_{z_{\ell}})$ are nonsingular. Therefore by taking $\sigma=\mathbold{ w}$, $\tau=\mathbold{ z}_{h}$, $\sigma_{1}=rev(\mathbold{ t}_{w})$, $\sigma_{2}=(\mathbold{ c}_{w},\mathbold{ t}_{w})$, $\tau_{1}=rev(\mathbold{ t}_{z})$ and $\tau_{2}=(\mathbold{ c}_{{z}_{h}},\mathbold{ t}_{z_{h}})$ and $\gamma=\mathbold{ w}$, $\delta=\mathbold{ z}_{h}$, $\gamma_{1}=rev(\mathbold{ t}_{w})$, $\gamma_{2}=(\mathbold{ c}_{w},\mathbold{ t}_{w})$, $\delta_{1}=rev(\mathbold{ t}_{z})$ and $\delta_{2}=(\mathbold{ c}_{{z}_{h}},\mathbold{ t}_{z_{h}})$, from Theorem \ref{gfprproof} it follows that $\mathbb{L}(\lambda)$ is a Rosenbrock linearization of $\mathcal{S}(\lambda)$ if the matrix assignments for $\mathbold{ c}_{w}$, $\mathbold{ c}_{z_{h}}$, $\mathbold{ c}_{v}$, and $\mathbold{ c}_{z_{\ell}}$ are nonsingular. By arguments similar to those given in Theorem \ref{tevenproof} we get the required result.
\end{proof}

\begin{example}
Let $G(\lambda)=D(\lambda)+B^{T}A(\lambda)^{-1}B\in\mathbb{C}(\lambda)^{r\times r}$ where $A(\lambda)=\lambda^{4}A_{4}+\lambda^{3}A_{3}+\lambda^{2}A_{2}+\lambda A_{1}+A_{0}\in\mathbb{C}[\lambda]^{n\times n}$, $D(\lambda)=\lambda^{5}D_{5}+\lambda^{4}D_{4}+\lambda^{3}D_{3}+\lambda^{2}D_{2}+\lambda  D_{1}+D_{0}\in\mathbb{C}[\lambda]^{r\times r}$, $C\in\mathbb{C}^{r\times n}$, and $B\in\mathbb{C}^{n\times r}$. Consider $h=2$, $\mathbold{w}=(1,2,0)$, $\mathbold{c}_{w}=(1)$, $\mathbold{z}_{h}=(-4,-3)$, $\mathbold{c}_{z_{h}}=(-4)$ and $\mathbold{t}_{w}=\phi=\mathbold{t}_{z_{h}}$. Similarly, $\ell=2$, $\mathbold{v}=(1,2,0)$, $\mathbold{c}_{v}=(1)$, $\mathbold{z}_{\ell}=-4,-3,-5)$, $\mathbold{c}_{z_{\ell}}=(-4)$, $\mathbold{t}_{v}=\phi=\mathbold{t}_{z_{\ell}}$. Also, $Q_{A}=\mathrm{diag}(I_{n},I_{n},I_{n},-I_{n})$ and $Q_{D}=\mathrm{diag}(I_{r},-I_{r},-I_{r},-I_{r},I_{r})$. Then, 
    {\footnotesize \[
     \mathbb{L}(\lambda)
       =\left[\begin{array}{cccc|ccccc}
         -A_{4} & \lambda A_{4} & 0 & 0 & 0 & 0 & 0 & 0 & 0 \\
         \lambda A_{4} & \lambda A_{3}+A_{2} & A_{1} & -I_{n} & 0 & 0 & 0 & 0 & 0 \\
         0 & A_{1} & -\lambda A_{1} + A_{0} & -\lambda I_{n} & 0 & 0 & 0 & B & 0 \\
         0 & I_{n} & -\lambda I_{n} & 0 & 0 & 0 & 0 & 0 & 0 \\\hline
         0 & 0 & 0 & 0 & 0 & -I_{r} & \lambda I_{r} & 0 & 0 \\
         0 & 0 & 0 & 0 & I_{r} & -\lambda D_{5} + D_{4} & -\lambda D_{4} & 0 & 0 \\
         0 & 0 & 0 & 0 & -\lambda I_{r} & -\lambda D_{4} & -\lambda D_{3} -D_{2} & - D_{1} & I_{r} \\
         0 & 0 & B^{T} & 0 & 0 & 0 & -D_{1} & \lambda D_{1}-D_{0} & -\lambda I_{r} \\
         0 & 0 & 0 & 0 & 0 & 0 & -I_{r} & \lambda I_{r} & 0 
      \end{array}\right]
     \]}
     is a skew-symmetric Rosenbrock linearization of $G(\lambda)$ whenever $A_{4}$ is nonsingular.
\end{example}

\printbibliography

\end{document}